\documentclass[12pt]{amsart}

\usepackage{latexsym,amsfonts,amsmath,amssymb,amsthm,url,graphics,psfrag,amscd,stmaryrd}
\usepackage[english]{babel}
\usepackage[latin1]{inputenc}
\usepackage[T1]{fontenc}
\usepackage{graphics}
\usepackage{graphicx}
\usepackage{color}

\newcommand{\E}{\mathbb{E}}
\newcommand{\Z}{\mathbb{Z}}
\newcommand{\Q}{\mathbb{Q}}
\newcommand{\pp}{\mathbb{P}}

\newcommand{\kA}{\mathcal{A}}
\newcommand{\kB}{\mathcal{B}}

\newcommand{\kO}{\mathcal{O}}

\newcommand{\kF}{\mathcal{F}}
\newcommand{\kG}{\mathcal{G}}

\newcommand{\kM}{\mathcal{M}}
\newcommand{\kE}{\mathcal{E}}
\newcommand{\kH}{\mathcal{H}}
\newcommand{\kI}{\mathcal{I}}
\newcommand{\kJ}{\mathcal{J}}

\newcommand{\kN}{\mathcal{N}}

\newcommand{\kL}{\mathcal{L}}
\newcommand{\kV}{\mathcal{V}}
\newcommand{\kU}{\mathcal{U}}

\newcommand{\pn}{\mathbb{P} }
\newcommand{\pnd}{\mathbb{P} }

\newcommand{\en}{\mathbb{E} }

\newcommand{\va}{\mathbb{V}\textrm{ar}}
\newcommand{\cov}{\textrm{Cov}}
\newcommand{\bin}{\mathrm{Bin}}

\newcommand{\lin}{\left[\kern-0.15em\left[}
\newcommand{\rin} {\right]\kern-0.15em\right]}
\newcommand{\linf}{[\kern-0.15em [}
\newcommand{\rinf} {]\kern-0.15em ]}
\newcommand{\ilin}{\left]\kern-0.15em\left]}
\newcommand{\irin} {\right[\kern-0.15em\right[}

\newtheorem{lem}{Lemma}[section]

\newtheorem{prop}[lem]{Proposition}
\newtheorem{theo}[lem]{Theorem}

\newtheorem {rem}[lem] {Remark}

\newtheorem*{ack}{Acknowledgments}

\title[Contact process on the preferential attachment graph]
       {\bf Metastability for the contact process on the preferential attachment graph}

\author{Van Hao Can}
\address{Aix Marseille Universit\'e, CNRS, Centrale Marseille, I2M, UMR 7373, 13453 Marseille, France}
\address{Institute of Mathematics, Vietnam Academy of Science and Technology, 18 Hoang Quoc Viet, 10307 Ha Noi, Viet Nam}
\email{cvhao89@gmail.com}


\begin{document}

\maketitle
\begin{abstract}
We consider the contact process on  the preferential attachment graph. The work of Berger,  Borgs,  Chayes and Saberi \cite{BBCS1}  confirmed  physicists' predictions that the contact process starting from a typical vertex   becomes epidemic for an arbitrarily small infection rate $\lambda$ with positive probability. More precisely, they showed that with probability $\lambda^{\Theta (1)}$, it survives for a time exponential in the largest degree. Here we obtain sharp bounds for the density of infected sites at a time close to exponential in the number of vertices (up to some logarithmic factor). In addition, a metastable result for the extinction time is also proved.
\end{abstract}

\section{Introduction}
The paper  aims at  proving  metastability results for the contact process on the preferential attachment random graph. We  improve  the  \cite{BBCS1}'s result in  two  aspects:  obtaining a better bound on the extinction time, and  estimating more accurately the density of the infected sites. Moreover, we also prove a metastable result for the extinction time.

\vspace{0,2cm}
The contact process is one of the most studied interacting particle systems, see in particular Liggett's book \cite{L}, and is also often interpreted as a model to describing how a virus  spreads  in a network. Mathematically, it can be defined as follows: given a locally finite graph $G=(V,E)$ and $\lambda >0,$ the contact process on $G$ with infection rate $\lambda$ is a Markov process $(\xi_t)_{t\geq 0}$ on $\{0,1\}^V.$ Vertices of $V$ (also called sites) are regarded as individuals which are either infected (state $1$) or healthy (state $0$). By considering $\xi_t$ as a subset of $V$ via $\xi_t \equiv \{v: \xi_t(v)=1\},$ the transition rates are given by
\begin{align*}
\xi_t \rightarrow \xi_t \setminus \{v\} & \textrm{ for $v \in \xi_t$ at rate $1,$ and } \\
\xi_t \rightarrow \xi_t \cup \{v\} & \textrm{ for $v \not \in \xi_t$ at rate }  \lambda \, \textrm{deg}_{\xi_t}(v),
\end{align*}
where $\textrm{deg}_{\xi_t}(v)$ denotes the number of edges between $v$ and other infected sites (Note that if $G$ is a simple graph, i.e. contains no multiple edges, then $\textrm{deg}_{\xi_t}(v)$ is just the number of infected neighbors of $v$ at time $t$). Given that $A \subset V,$ we denote by $(\xi_t^A)_{t \geq 0}$ the contact process with initial configuration $A$. If $A=\{v\}$ we simply write $(\xi^v_t)_{t\geq 0}$. 

Originally the contact process was studied on integer lattices or homogeneous trees. More recently, probabilists started  investigating this process on some families of random graphs like the Galton-Watson trees,  small world networks, configuration models, random regular graphs, and preferential attachment graphs, see for instance \cite{P, DJ, CD, CS, D, MVY, MV, BBCS1}.

\vspace{0,2cm}
The preferential attachment  graph (a definition will be given later) is well-known as a pattern of scale-free or social networks. Indeed, it not only shows  the power-law degree sequence of a host in real world networks, but also reflects a wisdom that the rich get richer - the newbies are more likely to get acquainted with more famous people rather than a relatively unknown person. Therefore there has been great interest in this random graph as well as the processes occurring on it, including the contact process. In \cite{BBCS1}, by introducing a new representation of the graph, the authors proved a remarkable result which validated  physicists predictions that the phase transition of the contact process occurs at $\lambda=0$. More precisely, they showed that there are positive constants $\theta$, $c$ and $C$, such that for all $\lambda$ small enough   
\begin{align} \label{rb}
\lambda^c \leq \pn\left(  \xi^u_{\exp(\theta \lambda^2 \sqrt{n})} \neq \varnothing \right) \leq \lambda^C,
\end{align}
where $(\xi^u_t)$ is the contact process starting from a uniformly chosen vertex.

\vspace{0,2 cm}
In this paper we will  improve \eqref{rb} as follows. 
\begin{theo} \label{td} Let $(G_n)$ be the sequential model of the preferential attachment graph with parameters $m \geq 1$ and $\alpha \in [0,1).$ Consider the contact process $(\xi_t)_{t\geq 0}$ with infection rate $\lambda >0$  starting from full occupancy on $G_n.$ Then there exist positive constants $c$ and $C$, such that for $\lambda$ small enough, 
\begin{equation} \label{g2}
 \pn\left(c\lambda^{1+ \frac{2}{\psi}} |\log \lambda |^{-\frac{1}{\psi}}  \leq \frac{|\xi_{t_n}|}{n} \leq C \lambda^{1+ \frac{2}{\psi}} |\log \lambda |^{-\frac{1}{\psi}} \right) =1-o(1) .
\end{equation}
where $\psi = \frac{1 - \alpha }{1 + \alpha}$ and $(t_n)$ is any sequence satisfying  $ t_n  \rightarrow \infty $ and $t_n \leq T_n = \exp \left(\frac{c \lambda^2 n}{(\log n) ^{1/\psi}}\right).$  
\end{theo}
By a well-known property of the contact process called self-duality (see \cite{L}, Section I.1) for any $t \geq 0$  we have 
\begin{align} \label{sdl}
 \sum \limits_{v \in V_n} 1(\{\xi^v_t \neq \varnothing\}) \mathop{ \Large =}^{(\kL)} |\xi_t| .
\end{align}
Therefore the survival probability as in \eqref{rb} is just the expected value of the density of infected sites as in  Theorem \ref{td}, so that our result is a stronger form of  the one in \cite{BBCS1}. Additionally we get a more precise estimate of the density and we allow $(t_n)$ to be larger. 

\vspace{0,2cm}
Since the empty configuration is the unique absorbing state of the contact process on finite graphs, the contact process on $G_n$ always dies out. Theorem \ref{td} shows that before dying, the contact process remains for a long time in a stationary situation in the sense that the density of infected sites is stable for a time streched exponential in the number of vertices. This metastable behavior for the contact process has been observed in some examples, such as the finite boxes (see \cite{L}, Section I.3), the configuration models (see \cite{CD,CS,MVY}), the random regular graphs (see \cite{LS}). 

\vspace{0.2 cm}

We also prove the following result.
\begin{prop} \label{cvelpa} Let $\tau_n$ be the extinction time of the contact process on the sequential preferential attachment graph with infection rate $\lambda>0$ starting from full occupancy. Then 
the following convergence in law holds
$$\frac{\tau_n}{\en(\tau_n)}\ \mathop{\longrightarrow}^{(\mathcal L)}_{n\to \infty}  \  \kE(1),$$
with $\kE(1)$ an exponential random variable with mean one. 
\end{prop}
In the proof of Proposition \ref{cvelpa}, see in particular \eqref{cxivxi} and \eqref{avbn}, we show  that after time $\exp (\log^2 n )$, either the contact process dies out, or it  equals  the contact process starting  from full occupancy (and thus the initial configuration is forgotten). Since $\exp (\log^2 n )$ is much smaller than the extinction time, this shows the presence of the metastability of the contact process. 

This kind of meatastability result  has been studied for different random processes, see for instance \cite{BH,OV} for the history and recent developments of this problem. In particular, it  has been proved for the contact process on finite boxes (see \cite{CGOV,M}), and finite homogeneous trees (see \cite{CMMV}).  Our method for proving Proposition \ref{cvelpa}  is rather general and only requires some simple hypothesis on the maximal degree and the diameter of the graph, which is satisfied in most scale-free random graphs models, like 
the configuration model with power law distribution,  or the preferential attachment graph.   
We refer  to Proposition \ref{cpcel}  for more details. 

\vspace{0.2 cm}

Now let us make some comments on the  proof of Theorem \ref{td}.

\noindent First, to obtain the time $T_n$ we will use the maintenance mechanism as in \cite{CD} instead of the one in \cite{BBCS1}. In the latter the authors used that in the preferential attachment graph the maximal degree is of order $\sqrt{n}$, plus the well-known fact that for any vertex $v$, the process survives a time exponential in the degree of $v$, once it is infected, yielding  \eqref{rb}. In the former, on the other hand, when considering the contact process on the configuration model, Chatterjee and Durrett employed many vertices with total degree of order $n^{1-\varepsilon}$, for any $\varepsilon >0$, and derived a much better bound on the extinction time. Here, our strategy is to find vertices with degree larger than $Cd(G_n)$, where  $C=C(\lambda)>0$ is a constant and $d(G_n)$ is the diameter of $G_n$, which is of order $\log n$. Thanks to Proposition 1 in \cite{CD}, we can deduce that the virus propagates along these vertices for a time exponential in  their total degree. Moreover, the degree distribution of the graph, denoted by {\bf p}, has a  power-law with exponent $\nu=  2+1/\psi $. Thus the number of these  vertices is of order $n (\log n)^{1- \nu}$ and their total degree is of order $n (\log n)^{-1/\psi}$, which explains the bound on $ t_n$ in Theorem \ref{td}.

It is worth noting  that for any graph with order $n$ edges (including $G_n$), the extinction time of the contact process is w.h.p. smaller than $\exp(Cn)$, for some $C>0$, see for instance Lemma 3.4 in \cite{C}. Hence our bound on $t_n$ is nearly optimal.

\vspace{0,2cm}
\noindent Secondly, to gain the precise estimate on the density, we use ideas from \cite{P, BBCS1, CD, MVY}: if the virus starting at a typical vertex wants to survive a long time, it has to infect a big vertex of degree  significantly larger than $\lambda^{-2}$. Then the virus is likely to survive in the neighborhood of this vertex for a time which is long enough to infect another big vertex, and so on. We can see that  the time required for a virus to spread from one big vertex to another is at least   $\lambda^{-\Theta(1)}$ (corresponding to the case when the distance between them is constant). Besides, it was shown that if $\deg(v) \geq K/ \lambda^2$, then the survival time of the contact process on the star graph formed by $v$ and it neighbors is about $\exp(cK)$. Hence the degree of big vertices should be larger than $C \lambda^{-2}|\log \lambda|$. Then we consider  $\Lambda$, the set of vertices which have a big neighbor. The probability for a vertex in $\Lambda$ to infect its big neighbor is of order $\lambda$. Moreover,  we will show in Section 4 that any big vertex has a positive probability to  make the virus survive up to time $T_n$. This means that the probability for the dual process starting from any vertex in $\Lambda$ to be active at time $T_n$ is of order $\lambda$. 
  Therefore the density of vertices from where the dual process survives up to $T_n$ is about $\lambda $ times the density of $\Lambda$. This is of order $\lambda \times \textbf{p} \big( [\lambda^{-2}|\log \lambda|, \infty \big ) \big) \lambda^{-2}|\log \lambda| \asymp \lambda^{1+ \frac{2}{\psi}} |\log \lambda |^{-\frac{1}{\psi}}$ yielding the desired lower bound.

We notice that the density of big vertices is of order $\textbf{p} \big( [\lambda^{-2}|\log \lambda|, \infty \big ) \big)$, which is $o(\lambda^{1+ \frac{2}{\psi}} |\log \lambda |^{-\frac{1}{\psi}})$. Thus it is  not optimal. Hence we need to consider also their neighbors. In fact this idea of using $\Lambda$ was first introduced in \cite{CD} for the configuration model.

 For the upper bound, we adapt the proof in \cite{MVY} for Galton-Watson trees, see the appendix.

\vspace{0,2 cm}
\noindent It is interesting to note also that if we consider the contact process on the configuration model with the same  power-law degree distribution,  the density is of order $\lambda^{1+ \frac{2}{\psi}} |\log \lambda |^{-\frac{2}{\psi}}$ (see \cite[Theorem 1.1]{MVY}), which is slightly smaller than the one in \eqref{g2}. This difference  is due to the fact that the distance between big vertices in the configuration model is about $|\log \lambda|$, instead of constant here. 

\vspace{0,2 cm}
\noindent Finally, the above strategy  works properly when studying the contact process on the P\'olya-point graph. In fact, proving the following result helped us pave the way to potential solutions for Theorem \ref{td}. 
\begin{prop} \label{prp}
Let $(\xi^{o}_t)$ be the contact process on the P\'olya-point graph with infection rate $\lambda >0$ starting from the root $o$. There exist positive constants $c$ and $C,$ such that for $\lambda$ small enough, 
$$c\lambda^{1+ 2/ \psi} |\log \lambda |^{-1/\psi} \leq \mathbb{P}(\xi^{o}_t \neq \varnothing \, \, \forall t) \leq C \lambda^{1+ 2/ \psi} |\log \lambda |^{-1/\psi}.$$
\end{prop}
 Theorem \ref{td} (resp. Proposition \ref{prp}) implies  that for all $\lambda >0$, the contact process becomes epidemic (resp. survives forever) with positive probability. We say that the critical values of the contact process on the preferential attachment graph and its weak limit are all zero.  This is a new  example of a more general  phenomena that  there is a relationship between the phase transition for the contact process on a sequence of finite graphs and the one on its weak local limit in the sense of  Benjamini--Schramm's convergence. Here are some known results on this topic: the contact process on the integer lattice $\mathbb{Z}^d$ and on finite boxes $\llbracket 1, n\rrbracket ^d$ exhibit a  phase transition at the same critical value $\lambda_c = \lambda_c(d)$, see \cite[Part I]{L} for all $d \geq 1$. The phase transition of the process on the random regular graph of degree $d$ and its limit, the homogeneous tree $\mathbb{T}_d$, occurs at the same constant $\lambda_1 (\mathbb{T}_d)$, see  \cite{MV}. The phase transition  of the contact process on $\mathbb{T}_d^{\ell}$ (the $d$-homogeneous tree of height $\ell$) and its limit, the canopy tree $\mathbb{C} \mathbb{T}_d$, happens at $\lambda_2(\mathbb{T}_d)$, see  \cite{CMMV, MV}. The critical value of the contact process on the configuration model with heavy tail degree distributions or on its limit, the Galton-Watson tree, is zero, see  \cite{P, CD, MVY, MMVY, CS}. 
 
 \vspace{0,2 cm}
Now the paper is organized as follows. In the next section,  based on \cite{BBCS2}, we give the definition of the sequential model of the preferential attachment graph as well as its weak local limit, the P\'olya-point graph. We also prove preliminary results on the graph structure and fix some notation. In Section 3, we prove Proposition \ref{prp}. The main theorem \ref{td} and Proposition \ref{cvelpa} are proved in  Sections 4 and 5 respectively.     
\section{Preliminaries}
\subsection{Construction of the random graph and notation.}
Let us give a definition following \cite{BBCS2} of the sequential model of the preferential attachment graph with parameters $m \geq 1$ and $\alpha \in[0,1)$. We  construct a sequence of graphs ($G_n$) with vertex set  $V_n =\{v_1, \ldots, v_n\}$ as follows.

First $G_1$ contains one vertex $v_1$ and no edge, and $G_2$ contains $2$ vertices $v_1, v_2$ and $m$ edges connecting them. Given $G_{n-1},$ we define $G_{n}$  the following way.  Add the vertex $v_n$  to the graph, and draw edges between $v_n$ and $m$ vertices $w_{n,1}, \ldots ,w_{n,m}$ (possibly with repetitions)  from $G_{n-1}$  as follows: with probability $\alpha ^{(i)}_n ,$ the vertex $ w_{n,i}$ is chosen uniformly at random from  $V_{n-1}$ where 
\begin{displaymath}
\alpha ^{(i)}_n = \left \{ \begin{array}{ll}
\alpha & \textrm{ if } i=1, \\
 \alpha \frac{2m(n-1)}{2m(n-2)+2m \alpha + (1- \alpha)(i-1)} = \alpha + \mathcal{O}(n^{-1}) & \textrm{ if } i \geq 2.
\end{array} \right.
\end{displaymath}
 Otherwise, $w_{n,i}=v_k$ with probability
$$ \frac{\textrm{deg}^{(i)}_{n-1}(v_k)}{Z^{(i)}_{n-1}},$$
where

$$\textrm{deg}^{(i)}_{n-1}(v_k) = \textrm{deg}_{n-1}(v_k) + \#\{1 \leq j \leq i-1: w_{n,j} =v_k \},$$
is the degree of $v_k$ before choosing $w_{n,i}$, and  
$$Z^{(i)}_{n-1} = \sum \limits_{k=1}^{n-1}\textrm{deg}^{(i)}_{n-1}(v_k) = \sum \limits_{k=1}^{n-1}\textrm{deg}_{n-1}(v_k)+i-1 = 2m(n-2)+ i-1,$$
with 
$\deg_{n-1}(v_k)$  the degree of $v_k$ in $G_{n-1}$.

This construction might seem less natural than in the  independent model where  with probability $\alpha$ we choose $w_{n,i}$  uniformly from $V_{n-1}$  and with  probability $1- \alpha$  it is chosen according to a simpler rule: $w_{n,i}=v_k $ with probability $  \deg_{n-1}(v_k)/2m(n-2)$.  However the sequential model constructed above is easier to analyze because it is exchangeable, and as a consequence it admits an alternative representation which contains more independence.  In \cite{BBCS2}, the authors called it the P\'olya urn representation which we now recall in the following theorem. To this end, we denote by $\beta(a,b)$  the Beta distribution, whose density is proportional to $x^{a-1}(1-x)^{b-1}$ on $[0,1]$, and by $\Gamma(a,b)$ the Gamma distribution,  whose density is proportional to $x^{a-1} e^{-bx}$ on $[0, \infty)$. For any $a<b$, $\kU([a,b])$ stands for the uniform distribution on  $[a,b]$.  
\begin{theo} \cite[Theorem 2.1]{BBCS2}\label{tb}
Fix $m \geq 1$, $\alpha \in [0,1)$ and $n \geq 1$. Set  $r=\alpha/(1- \alpha), \psi_1=1$, and let $\psi_2,\ldots,\psi_n$ be independent random variables with  law
$$\psi_j \sim \beta (m+2mr, (2j-3)m+ 2mr(j-1)).$$
Define
$$\varphi_j = \psi_j \prod \limits_{t=j+1}^n(1- \psi_t), \quad \quad S_k = \sum \limits_{j=1}^k \varphi_j, \, \textrm{ and } \quad I_k = [S_{k-1},S_k).$$
Conditionally on $\psi_1,\ldots,\psi_n$, let $\{U_{k,i}\}_{k=1,\ldots,n, i=1,\ldots,m}$ be a sequence of independent random variables, with $U_{k,i} \sim \kU([0,S_{k-1}])$. Start with the vertex set $V_n=\{v_1, \ldots, v_n\}$. For $j< k$, join $v_j$ and $v_k$ by as many edges as the number of indices $ i \in \{1,\ldots,m\},$ such that  $U_{k,i} \in I_j$. Denote the resulting random graph by $G_n.$ 

Then $G_n$ has the same distribution as the sequential model of the preferential attachment graph.
\end{theo}
We remark that in \cite{BBCS2}, the authors state the theorem for $m\geq 2$. However, their proof also works properly when $m=1$. 

 From now on, we always consider the random multi-graph $G_n$ constructed as in this theorem.

\vspace{0,2 cm}
We now look at the local structure of $G_n$. It was shown in \cite{BBCS2} that $G_n$ is locally tree-like, with some subtle degree distribution that we now recall.  First we define some constants:
$$\chi = \frac{1+2r}{2+2r} \quad \quad \textrm{ and } \quad \quad \psi = \frac{1- \chi}{\chi} = \frac{1}{1+ 2 r}.$$
Note that $1/2 \leq \chi <1$ and $0 < \psi \leq 1$.  Let $F  \sim \Gamma(m+2mr,1)$ and $F' \sim\Gamma(m+2mr+1,1)$.

 We will construct inductively a random rooted tree $(T,o)$ with vertices identified with elements of $\cup_{\ell \geq 1} \mathbb{N}^{\ell}$ (where vertices at generation $\ell$ are elements of $\mathbb{N}^{\ell}$) and a map which associates to each vertex $v$ a position $x_v$ in $[0,1].$ Additionally each vertex (except the root) will be assigned a type, either R or L.  
\begin{itemize}
\item[$\bullet$] The root $o=(0)$ has position $x_o= U_0^{\chi}$, where $U_0 \sim \kU([0,1])$.
\item[$\bullet$] Given $v \in T$ and its position $x_v$, define  
\begin{displaymath}
m_v= \left \{ \begin{array}{ll}
m & \textrm{ if $v$ is the root or of type L},\\
m-1 & \textrm{ if $v$ is  of type R}.
\end{array} \right.
\end{displaymath}
and
\begin{displaymath}
\gamma_{v} \sim \left \{ \begin{array}{ll}
F & \textrm{ if $v$ is the root or of type R},\\
F' & \textrm{ if $v$ is  of type L}.
\end{array} \right.
\end{displaymath}
The children of $v$ are $(v,1)$,\ldots,$(v,m_v)$, $(v,m_v+1)$,\ldots,$(v,m_v+q_{v})$, the first $m_v$'s  are of type $L$ and the remaining ones are of type $R$. Conditionally on $x_v$,  $x_{(v,1)},\ldots,x_{(v,m_v)}$ are i.i.d. uniform  random variable in  $[0,x_v]$, and  $x_{(v,m_v+1)},\ldots,$ $ x_{(v,m_v+q_v)}$ are the points of the Poisson point process on $[x_v,1]$ with intensity
$$\rho_v(x) dx = \gamma_v\frac{\psi x^{\psi-1}}{x_{v}^{\psi}} dx.$$

This procedure defines inductively an infinite rooted tree $(T,o)$, which is called the P\'olya-point graph and  $(x_v)_{v \in T}$ is called the P\'olya-point process. 
\end{itemize}
For any vertex $v$ in a graph $G$ and any integer $R$, we call $B_G(v,R)$ the ball of radius $R$ around $v$ in $G$, which contains all vertices at distance smaller than or equal to $R$ from $v$ and all edges connecting them. 
\begin{theo} \cite[Theorem 2.2]{BBCS2}\label{t2.3}
Assume that the random graph $G_n$ is constructed as in Theorem \ref{tb}. Let $u$ be a  vertex  chosen uniformly at random in $G_n$ and let $R$ be some fixed constant. Then $B_{G_n}(u, R)$ converges weakly to the ball $B_T(o,R)$ in the P\'olya-point graph.  
 \end{theo}
Note that in \cite{BBCS2} the authors prove this theorem for $m \geq 2$. For $m=1$, recently, in \cite{BMR} the authors show that the local limit of preferential attachment graph is always the P\'olya-point graph regardless the initial (seed) graph (in our case, the initial graph contains two vertices $v_1, v_2$ and one edge connecting them). 

\vspace{0,2 cm}
Now we introduce some notation. We call $\pn$ a probability measure on a space in which the random graph $G_n$ is defined together with the contact process. Since we will fix $\lambda$, we  omit it in the notation. We also call $\mathbb{P}$ a probability measure on a space in which the P\'olya-point graph as well as the contact process are defined.

\vspace{0,2 cm}
 We denote the indicator function of a set $A$ by ${\bf 1}(A)$.  For any  vertices $v $ and $w$ we write $v \sim w$ if there is an edge between them (in which case we say that they are neighbors or connected),  and $v \not \sim w$ otherwise. We denote by $|G|$ the order of  $G$.
 
A graph in which all vertices have degree one, except one which is connected to all the others is called a {\bf star graph}. The only vertex with degree larger than one is called the center of the star graph, or central vertex. 

\vspace{0,2 cm}
If $f $ and $g$ are two real functions, we write $f= \mathcal{O}(g)$ if there exists a constant $C>0,$ such that $f(x) \leq C g(x)$ for all $x ;$ $f \gtrsim g$ (or equivalently $g \lesssim f$) if $g=\kO(f)$; $f \asymp g $ if $f= \mathcal{O}(g)$ and $g= \mathcal{O}(f);$  $f=o(g)$ if $f(x)/g(x) \rightarrow 0$ as $x \rightarrow \infty.$   Finally  for a sequence of r.v.s $(X_n)$ and a function $f : \mathbb{N} \rightarrow (0, \infty)$, we say that $X_n \asymp f(n)$ holds w.h.p. if there exist positive constants $c$ and $C,$ such that $\pn(c f(n) \leq X_n \leq C f(n)) \rightarrow 1$.
  
\subsection{Preliminary results on the random graph.}
We first recall a version of the Azuma-Hoeffding inequality for martingales which we will use throughout this paper (see for instance \cite{CL}). 
\begin{lem}\label{l1}
Let $(X_i)_{i\geq 0}$ be a martingale satisfying  $|X_i -X_{i-1}| \leq 1$ for all $i \geq 1$. Then for any $n$ and $t>0$, we have
$$\mathbb{P} \left( |X_n-X_0| \geq t \right) \leq 2\exp(-t^2/2n).$$
\end{lem}
\noindent From this inequality we can deduce a  large deviations result.  Let $(X_i)_{i\geq 1}$ be a sequence of independent Bernoulli random variables. Assume that $0< 2p \leq \en (X_i) \leq Mp $ for all $i$.  Then there exists $c = c(M)>0,$ such that for all $n$
\begin{align} \label{ldv}
\mathbb{P} \left( np \leq \sum \limits_{i=1}^n X_i \leq 2Mnp \right) \geq 1 - \exp(- c np).
\end{align} 
\noindent Now we present some estimates on the sequences $(\varphi_i)$, $(\psi_j)$ and $(S_k)$.
\begin{lem}\label{l2}
Let $(\varphi_i)_i$, $(\psi_j)_j$ and $(S_k)_k$ be sequences of random variables as in Theorem $\ref{tb}.$ Then there exist positive constants $\mu$ and $\theta_0$, such that for all $\theta \leq \theta_0$, the following assertions hold.
\begin{itemize}
\item[(i)] $ \en (\psi_j) = \frac{\chi}{j} +\kO(\frac{1}{j^2}), \quad \en(\psi_j^2) \asymp \frac{1}{j^2}.$ \\
\item[(ii)] For any $\varepsilon >0,$ there exists $K=K(\varepsilon)< \infty,$ such that 
$$\pn(\mathcal{E}_{\varepsilon}) \geq 1- \varepsilon,$$
where 
$$\mathcal{E}_{\varepsilon}= \Big \{ \Big| \frac{S_j}{S_k}  -  \left(\frac{j}{k}\right)^{\chi} \Big| \leq \varepsilon (j/k)^{\chi} \quad \forall \, K(\varepsilon) \leq j \leq  k \leq n\Big \}.$$ 
\item[(iii)] As $n$ tends to infinity, 
\[\pp(i \psi_i \leq 2 \log n \, \, \forall \, i=1, \ldots, n)=1-o(1).\]
\item[(iv)] $\pn \left(\mu/j \geq \psi_j \geq \theta/j \right) \geq 2 \theta.$ \\
\item[(v)] $\en (\varphi_j 1( \psi_j \geq \theta/j)) \geq \theta \en(\varphi_j)$. 
\end{itemize}
\end{lem}
\begin{proof}
Let us start with Part (i). Observe that if $\psi \sim \beta(a,b)$, then
$$\mathbb{E}(\psi)= \frac{a}{a+b} \quad \textrm{ and } \quad \mathbb{E}(\psi^2)= \frac{a(a+1)}{(a+b)(a+b+1)}. $$
Hence the result follows from the fact that $\psi_j \sim \beta( m+2mr, (2j-3)m+2mr(j-1)).$ 

Part (ii) is  a direct consequence of Lemma 3.1 in \cite{BBCS2}.  We now prove (iii). Since $\psi_i \sim \beta(a,b_i)$ with $a= m +2mr$ and $b_i=(2m+2mr)i -(3m +2mr)$, we have for $i\geq 2 \log n$
\begin{eqnarray*}
\pp \left(\psi_i > \frac{2 \log n}{i}\right) &=& \frac{1}{B(a,b_i)} \int_{2 \log n/i}^1 x^{a-1} (1-x)^{b_i -1} dx \\
& \lesssim & b_i^{a} \left(1-\frac{2 \log n}{i}\right)^{b_i} \\
& \lesssim & n^{-2}.
\end{eqnarray*}
Here, we have used that $B(a,b) \asymp \kO(b^{-a})$ when $a$ is fixed. On the other hand,  when $i< 2 \log n$, this probability is zero. Therefore 
\[\pp(i \psi_i \leq 2 \log n \, \, \forall \, i=1, \ldots, n)=1-o(1).\]
For Part (iv), Chebyshev's inequality gives that for any $\delta \in (0,1)$
$$\pn \left(|\psi_j - \en(\psi_j)|> (1- \delta) \en(\psi_j) \right) \leq \frac{\mathbb{V}\textrm{ar}(\psi_j)}{(1-\delta)^2 \en(\psi_j)^2}.$$
Moreover, if $\psi \sim \beta(a,b),$ then 
$$ \frac{\va(\psi)}{ \mathbb{E}(\psi)^2}= \frac{\mathbb{E}(\psi^2)}{\mathbb{E}(\psi)^2}-1 = \frac{(a+1)(a+b)}{a(a+b+1)}-1 = \frac{b}{a(a+b+1)} \leq \frac{1}{a}.$$
Therefore for any $j$
$$\pn(\psi_j \in (\delta \en(\psi_j), (2-\delta)\en(\psi_j))) \geq 1- \frac{1}{(1-\delta)^2(m+2mr)}.$$ 
Hence thanks to (i) we can choose positive constants $\mu$ and $\theta$, such that for all $j$
\begin{equation} \label{e3}
\pn \left(\psi_j \in \left( \frac{\theta }{j}, \frac{\mu}{j}\right) \right) \geq 2\theta.
\end{equation}
For (v), we notice that by (i), $\en(\psi_j) \asymp 1/j$. Hence
 \begin{align} \label{rvm}
 \en(\psi_j 1(\psi_j \geq \theta /j)) \geq \en(\psi_j) - (\theta /j) \geq c\en(\psi_j),
\end{align}  
for some constants $c$ and $\theta$ independent of $j$. Moreover, using the fact that these random variables $(\psi_j)_{j\geq 0}$ are independent, we obtain 
\begin{align*}
\en(\varphi_j 1(\psi_j \geq \theta /j)) & = \en \left(\psi_j 1( \psi_j \geq \theta/j)\prod_{t=j+1}^n (1- \psi_t)\right) \\
& =  \en (\psi_j 1( \psi_j \geq \theta/j)) \en \left(\prod_{t=j+1}^n (1- \psi_t)\right) \\
& \geq  c \en(\psi_j)\en \left(\prod_{t=j+1}^n (1- \psi_t)\right) \\
& =  c \en(\varphi_j). 
\end{align*}
Here for the third line, we have used \eqref{rvm}.
\end{proof}
The preferential attachment graph is known as a prototype of small world networks whose  diameter and typical distance (the distance between two randomly chosen vertices) are of logarithmic order. In fact, these quantities  in the independent model  were well-studied, see for instance \cite{DVH} or \cite{V}.  In the following two lemmas, we  prove  similar estimates for the sequential model. These estimates are in fact weaker but  sufficient for our purpose.
\begin{lem} \label{lmdi}
Let $d(G_n)$ be the diameter of the random graph $G_n,$ i.e. the maximal distance between any pair of vertices in $G_n$. Then there exists a positive constant $b_1$, such that
$$\pn( d(G_n) \leq b_1 \log n) = 1 - o(1).$$
\end{lem}
\begin{proof}
 Let $\varepsilon \in (0,1/2)$ be given, and recall the definitions of $K(\varepsilon)$ and  $\kE_{\varepsilon}$ given in Lemma \ref{l2} (ii). We first bound $d(v_1,v_n)$. Define a decreasing random sequence $(n_i)_{i \geq 0}$ as follows $n_0=n$  and for $i \geq 1$,  $n_{i+1}$ is arbitrarily chosen from indices such  that $v_{n_{i+1}}$ receives an edge emanating from $v_{n_i}$.   Define
$$X_i=1(\{n_{i} \leq n_{i-1}/2\}) \textrm{ and } \kF_i = \sigma(n_j: j \leq i) \vee \sigma((\varphi_t)).$$
By the construction of the graph in Theorem \ref{tb}, we have
\[\en(X_{i+1}\mid \kF_i) = \pn(n_{i+1} \leq n_i/2 \mid  \kF_i) = \frac{S_{[n_i/2]}}{S_{n_i-1}}, \]
with $(S_i)$ as in Theorem \ref{tb}. We now define  
$$ \sigma_n = \inf \{ i: n_{i+1} \leq \log n \}.$$
If $i \leq \sigma_n$, then $n_i > \log n > 2K(\varepsilon)$ for $n$ large enough. Therefore by Lemma \ref{l2} (ii), we have on $\kE_{\varepsilon}$ for $i \leq \sigma_n$
\begin{align*}
 \frac{S_{[n_i/2]}}{S_{n_i-1}} \geq (1- \varepsilon)\left(\frac{[n_i/2]}{ n_i-1} \right)^{\chi} \geq \frac{1}{2^{\chi +1}} =:p.
\end{align*}
In other words,  we have
\begin{align} \label{sao} 
\en (X_{i+1} \mid \kF_{i}, \kE_{\varepsilon}) 1(i \leq \sigma_n)\geq p 1(i \leq \sigma_n).
\end{align}
Let $$Y_k = \sum \limits_{i=1}^{k} \left(  X_i - \en (X_i \mid \kF_{i-1},\kE_{\varepsilon}) \right).$$
Then $(Y_k)$ is a martingale with respect to the filtration $(\kF_k)$ and $|Y_k-Y_{k-1}| \leq 1$. By using Lemma \ref{l1} we get that 
\begin{align} \label{byk}
\pn \left( Y_k \leq -kp/2 \mid (\varphi_t)\right) \leq  2\exp(-kp^2/8).
\end{align}
Note that  there is a slight abuse of notation: the left-hand side of \eqref{byk} should be a conditional expectation of an indicator variable, but for simplicity we just write it as a conditional probability. Now,  it follows from \eqref{sao} and \eqref{byk}  that 
\begin{align} \label{azh}
\pn \left( \sum \limits_{i=1}^{k}  X_i  \leq kp/2 \mid (\varphi_t), \kE_{\varepsilon},  \sigma_n \geq k\right) \leq  2\exp(-kp^2/8).
\end{align}
Moreover, if  $ \sum _1^k X_i > (\log_2 n - \log_2 \log n) $, then $n_k \leq \log n$ or equivalently $\sigma_n \leq k$. Hence for all $C=C(p)$ large enough
\begin{align*}
&\pn \left( \sigma_n \geq C \log n, \sum \limits_{i=1}^{C \log n}  X_i  > (C\log n)p/2 \mid (\varphi_t), \kE_{\varepsilon} \right) \\
\leq &\pn \left( \sigma_n \geq C \log n, \sum \limits_{i=1}^{C \log n}  X_i  > \log_2 n - \log_2 \log n \mid (\varphi_t), \kE_{\varepsilon} \right) = 0
\end{align*}
Therefore, we  have
\begin{align} \label{lbs}
\pn (\sigma_n \geq C \log n \mid (\varphi_t), \kE_{\varepsilon}) & =   \mathbb{P}_n \left( \sigma_n \geq C \log n, \sum \limits_{i=1}^{C \log n}  X_i  \leq (C\log n)p/2 \mid (\varphi_t), \kE_{\varepsilon} \right) \notag \\
 &\leq 2 \exp (-(C \log n) p^2/8))= \mathcal{O}(n^{-2}).
\end{align}
Here for the last line, we have used \eqref{azh} for  $k= C \log n$, with some $C$ large enough.  On the other hand,
\[\{\sigma_n \leq C \log n\} \subset\{\exists \, k \leq \log n: d(v_n,v_k) \leq C \log n\},\]
and  $d(v_1,v_k) \leq \log n$ for all $k \leq \log n$. Therefore,
\[\{\sigma_n \leq C \log n\} \subset\{ d(v_n,v_1) \leq (C+1) \log n\}.\]
Hence by \eqref{lbs}, we obtain that 
\begin{equation} \label{hta}
\pn(d(v_1,v_n) \geq (C+1) \log n \mid (\varphi_t), \kE_{\varepsilon}) =  \mathcal{O}(n^{-2}).
\end{equation}

\noindent Let $d_{G_k}(v_i,v_j)$ be the distance between $v_i$ and $v_j$ in  $G_k$ for $i,j \leq k \leq n$.  Note that 
$$d_{G_k}(v_i,v_j) \geq d(v_i,v_j) = d_{G_n} (v_i,v_j). $$
 Similarly to \eqref{hta}, we  deduce that on $\kE_{\varepsilon}$, for all $i \geq C \log n$,
\begin{align*}
\pn(d(v_1,v_i) \geq (C+1) \log i \mid (\varphi_t)) \leq \mathbb{P}(d_{G_i}(v_1,v_i) \geq (C+1) \log i \mid (\varphi_t)) = \kO(i^{-2}).
\end{align*}
Hence on $\kE_{\varepsilon}$
$$\pn( d(v_1,v_i) \leq (C+1) \log n  \, \,\forall \, i \geq C \log n \mid (\varphi_t)) = 1-o(1).$$
Therefore by taking expectation with respect to $(\varphi_t)$ and using Lemma \ref{l2} (ii), we get
$$\pn (d(G_n) \leq 2(C+1) \log n ) \geq  1-2 \varepsilon,$$
which proves the result by letting $\varepsilon$ tend to $0$.
\end{proof}
Before proving the lower bound on the typical distance, we make a remark which will be used frequently in this paper. It follows from the definition of $G_n$ that for all $i<j$,
$$\pn(v_i \not \sim v_j \mid (\varphi_t))= \left(1 - \frac{\varphi _i}{S_{j-1}} \right)^m.$$
Hence
$$\frac{ \varphi_i}{S_{j-1}} \leq \pn(v_i  \sim v_j \mid (\varphi_t)) \leq \frac{m \varphi_i}{S_{j-1}}.$$
Then by using the following identities
$$S_{j-1} = \sum \limits_{t=1}^{j-1} \varphi_t = \prod \limits_{t=j}^{n}(1- \psi_t) \quad \textrm{ and } \quad \varphi_i = \psi_i \prod \limits_{t=i+1}^{n} (1- \psi_t), $$
we obtain that
\begin{equation} \label{ecp}
\frac{\psi_i S_i}{S_{j-1}} \leq \pn(v_i  \sim v_j \mid (\varphi_t)) \leq m \frac{\psi_i S_i}{S_{j-1}}.
\end{equation} 
\begin{lem} \label{ltyd} Let $w_1$ and $w_2$ be two uniformly chosen vertices from $V_n$. Then there is a positive constant $b_2$, such that w.h.p. 
$$d(w_1,w_2) \geq \frac{b_2 \log n}{\log  \log n}.$$
\end{lem}
\begin{proof}  Fix  $\varepsilon \in (0,1/2)$ a positive constant. Then we define 
\[\kI_{\varepsilon} = \kE_{\varepsilon} \cap \{i \psi_i \leq 2 \log n \, \,\forall \, i =1, \ldots, n\},\]
with $\kE_{\varepsilon}$ as in Lemma \ref{l2}. It follows from \ref{l2} (ii) and (iii) that for all $n$ large enough 
\begin{equation} \label{piep}
\pp(\kI_{\varepsilon}) \geq 1 -2 \varepsilon.
\end{equation}
We now use an argument from \cite[Lemma 7.16]{V} to bound the typical distance.  We call a sequence of distinct vertices $\pi=(\pi_1, \ldots, \pi_k)$  a self-avoiding path. We write  $\pi \subset G_n$ if $\pi_i$ and $\pi_{i+1}$ are neighbors for all $1 \leq i \leq k-1$. Let $\Pi(i,j,k)$ be the set of all self-avoiding paths of length $k$ starting from $v_i$ and ending at $v_j$. We then claim that for all $i,j,k \geq 1$, 
\begin{itemize}
\item[(i)] $\pn (d(v_i,v_j)=k \mid (\varphi_t), \kI_{\varepsilon}) \leq g_k(i,j):= \sum \limits_{\pi \in \Pi(i,j,k)} \pn (\pi \subset G_n \mid (\varphi_t), \kI_{\varepsilon})$, \\
\item[(ii)] $g_{k+1}(i,j)  \leq \sum \limits_{s \neq i,j}  g_1(i,s) g_{k}(s,j)$.
\end{itemize}
The first claim is clear, because if $d(v_i,v_j)=k$ then there exists a self-avoiding path in $\Pi(i,j,k)$ which is in $G_n$. For the second one, we note that for any self-avoiding path $\pi=(\pi_1,\ldots, \pi_k)$,
$$
\pn(\pi \subset G_n \mid (\varphi_t), \kI_{\varepsilon}) = \pn(\pi_1 \sim \pi_2 \mid (\varphi_t), \kI_{\varepsilon}) \times \pn(\bar{\pi} \subset G_n \mid (\varphi_t), \kI_{\varepsilon}),
$$
where $\bar{\pi}=(\pi_2, \ldots, \pi_k)$. Indeed, if $j < k$, then the event that $v_j \sim v_k$ depends only on the $(U_{k,i})_{i\leq m}$. Hence  this result follows from  the facts that the vertices in $\pi$ are distinct and that the  $\{(U_{k,i})_{i\leq m} \}_{k}$ are independent. We  are now in position to  prove (ii):
\begin{align*}
g_{k+1}(i,j)& = \sum \limits_{s \neq i,j} \sum  \limits_{ \substack  { v_i \not \in \bar{\pi} \\ \bar{\pi} \in \Pi(s,j,k) }} \pn (v_i \sim v_s, \bar{\pi} \subset G_n \mid (\varphi_t), \kI_{\varepsilon}) \\
& \leq \sum \limits_{s \neq i,j} \sum  \limits_{  \bar{\pi} \in \Pi(s,j,k) } \pn (v_i \sim v_s \mid (\varphi_t), \kI_{\varepsilon}) \times \pn( \bar{\pi} \subset G_n \mid (\varphi_t), \kI_{\varepsilon}) \\
& = \sum \limits_{s \neq i,j} g_1(i,s) g_k(s,j).
\end{align*}
We prove by induction on $k$ that  there is a positive constant $C$, such that
\begin{align} \label{bkij}
g_k(i,j) \leq \frac{(C\log n)^{2k-1}}{\sqrt{ij}}.
\end{align}
 For $k=1$, it follows from  \eqref{ecp} that  for all  $i<j$, 
\begin{align} \label{cuc}
g_1(i,j) & = \pn(v_i \sim v_j \mid (\varphi_t), \kI_{\varepsilon})\notag\\
& \leq m  \en \left( \frac{\psi_i S_i}{S_{j-1}} \mid (\varphi_t), \kI_{\varepsilon}\right).
\end{align}
We now claim that on $\kI_{\varepsilon}$ 
\begin{equation} \label{pssj}
\frac{\psi_i S_i}{S_{j-1}}\leq \frac{4 \log n}{\sqrt{ij}}.
\end{equation}
Indeed, we recall that on $\kI_{\varepsilon}$,
\[ \psi_i \leq \frac{2 \log n}{i} \,\, \forall \, i =1, \ldots, n \quad \textrm{and } \quad \Big |  \frac{S_i}{S_j} - \left(\frac{i}{j} \right)^{\chi} \Big |\leq  \varepsilon \left(\frac{i}{j} \right)^{\chi} \, \, \forall j \geq i \geq K(\varepsilon),\]
with $K(\varepsilon)$ as in Lemma \ref{l2} (ii). 

If $i<j \leq K(\varepsilon)$ then \eqref{pssj} holds for all $n$ large enough (the left hand-side is bounded by $1$ and the right-hand side tends to infinity). 
 If $j>i\geq K(\varepsilon)$ then 
\begin{eqnarray} \label{sisj}
 \psi_i \frac{S_i}{S_{j-1}} \leq \frac{2 \log n}{i} \times (1+ \varepsilon)\left(\frac{i }{j-1}\right)^{\chi} \leq \frac{4 \log n}{ \sqrt{ij}},
\end{eqnarray} 
since $i<j$ and $\chi \geq 1/2 > \varepsilon$. If $i \leq K(\varepsilon)<j$ then using  $\psi_i \leq 1$ and the fact that the sequence $(S_i)$ is increasing, we obtain
\begin{eqnarray*}
\psi_i \frac{S_i}{S_{j-1}} &\leq&   \frac{S_{K(\varepsilon)}}{S_{j-1}}  \leq   (1+ \varepsilon)\left(\frac{K(\varepsilon) }{j-1}\right)^{\chi}  \leq \frac{4 \log n}{ \sqrt{ij}},
\end{eqnarray*} 
for all $n$ large enough. In any case, \eqref{pssj} holds.   It follows from \eqref{cuc} and \eqref{pssj} that \eqref{bkij} holds for $k=1$.  

Assume now that \eqref{bkij} is true for some $k$, and let us prove it for $k+1$. By using the induction hypothesis and (ii), we get that
\begin{align*}
 g_{k+1}(i,j) & \leq \sum \limits_{s \neq i,j}  g_1(i,s) g_{k}(s,j) \\
 & \leq \sum \limits_{s \neq i,j} \frac{C \log n}{\sqrt{is}} \frac{(C\log n)^{2k-1}}{\sqrt{sj}} \\
 & \leq \frac{(C\log n)^{2k+1}}{\sqrt{ij}},
\end{align*} 
which proves the induction step. Now it follows from  (i) and \eqref{bkij} that 
\begin{align*}
\pn (d(w_1,w_2 ) \leq K \mid (\varphi_t), \kI_{\varepsilon}) & \leq \frac{1}{n^2} \sum \limits_{k=1}^K \sum \limits_{1 \leq i,j \leq n} g_k(i,j) \\
& \leq \frac{1}{n^2} \sum \limits_{k=1}^K \sum \limits_{1 \leq i,j \leq n} \frac{(C \log n)^{2k-1}}{\sqrt{ij}} \\
& \leq \frac{(C \log n)^{2K}}{n}.
\end{align*}
Therefore if  $K= \log n / (3C \log \log n)$, then  $\pn (d(w_1,w_2 \mid (\varphi_t), \kI_{\varepsilon} ) \leq K) =o(1)$. Combining this with \eqref{piep} gives the desired lower bound with $b_2=1/(3C)$. 
\end{proof}
\begin{rem}
\emph{ The  bound in Lemma \ref{ltyd} is probably not sharp. Indeed for the independent model, it is proved in \cite{DVH} or \cite{V}  that if $\alpha >0$ (or equivalent $\chi >1/2$), then  w.h.p. $d(w_1,w_2) \geq c \log n$; and otherwise, w.h.p. $d(w_1,w_2) \geq  \log n/ (C+ \log \log n)$, for some positive constants $c$ and $C$.
}
\end{rem}
\subsection{Graphical construction of the contact process}
We briefly recall here the graphical construction of the contact process (see more in Liggett's book \cite{L}).

Fix $\lambda>0$ and an oriented graph $G$ (recall that a non-oriented graph can also be seen as oriented by associating 
to each edge two oriented edges). Then assign independent Poisson point processes $D_v$ of rate $1$ to each 
vertex $v \in V$ and $D_{e}$ of rate $\lambda$ to each  oriented edge $e$. 
Set also $D_{(v,w)}:=\cup_{e : v\to w}\, D_e$,  
for each ordered pair $(v,w)$ of vertices, where the notation $e: v\to w$ means that the oriented edge $e$ goes from $v$ to $w$.

We say that there is an infection path from $(v,s)$ to $(w,t)$, and we denote it by 
\begin{eqnarray}
\label{vswt}
(v,s)\longleftrightarrow (w,t), 
\end{eqnarray}
either if $s=t$ and $v=w$, or if $s<t$ and if there is a sequence of times $s=s_0< s_1<\ldots<s_l<s_{l+1}=t,$ and a sequence of vertices $v=v_0,v_1,\ldots,v_l=w$ such that for every $i=1,\ldots,l$ 
\begin{displaymath}
\left \{ \begin{array}{ll}
 s_i \in D_{(v_{i-1},v_i)} \quad \textrm{ and }\\
 D_{v_i} \cap [s_i, s_{i+1}] = \varnothing.
\end{array} \right.
\end{displaymath}
Furthermore, for any $A$, $B$ two subsets of $V_n$ and $I$, $J$ two subsets of $[0,\infty)$, we write 
$$A\times I \longleftrightarrow B\times J,$$
if there exists $v\in A$, $w\in B$, $s\in I$ and $t\in J$, such that \eqref{vswt} holds. 
Then for any $A\subset V_n$, the contact process with initial configuration $A$ is defined by  
$$\xi^A_t :=\left\{v \in V_n:  A\times\{0\}\longleftrightarrow (v,t)\right\},$$
for all $t\ge 0$.  
It is well known that $(\xi^A_t)_{t \geq 0}$ has the same distribution as the process defined in the introduction. Moreover, this graphic construction allows us to couple contact processes with different initial states with additivity property: for all $t\geq 0$ and $A,B \subset V$, 
\begin{equation*}
\xi^{A \cup B}_t = \xi^A_t \cup \xi^B_t. 
\end{equation*}

\subsection{Contact process on star graphs} We will see that star graphs play a crucial role in the conservation of the virus on the preferential attachment graph. Hence, it is important to  understand the behavior of the contact process on a single star graph as well as the transmission between them. These have been studied for a long time by many authors, for instance in \cite{P, BBCS1, CD, MVY}. The results we need will be summarized  in  Lemma \ref{lb} and \ref{lst} below. We say that a vertex $v$ is $\textbf{lit}$ (the term is taken from \cite{CD}) at some time $t$ if the proportion of its infected neighbors at time $t$ is larger than $\lambda/(16e)$.
\begin{lem} \label{lb} Let  $(\xi_t)$ be the contact process  on a star graph $S$ with center $v$. There exists a positive constant $c^*$, such that  the following assertions hold.
\begin{itemize}
\item[(i)]  $\mathbb{P}\left( v \textrm{ is lit at time } 1 \mid \xi_0(v)=1\right) \geq (1- \exp(-c^*\lambda |S|))/e. $  \\

\item[(ii)] $\mathbb{P}\left( \exists t >0: v \textrm{ is lit at time t} \mid \xi_0(v)=1\right) \rightarrow 1 \quad$ as $|S| \rightarrow \infty$. \\

\item[(iii)] If $ \lambda ^2 |S| \geq 64 e^2$, and $v$ is lit at time $0$, then $v$ is lit during the time interval  $[ \exp(c^* \lambda^2 |S|), 2\exp(c^* \lambda^2 |S|)] $ with probability larger than $1- 2 \exp (-c^* \lambda ^2 |S|).$ 
\end{itemize} 
\end{lem}
\begin{proof}
Parts (i) and (ii) are exactly Lemma 3.1 (i), (iii)   in \cite{MVY}. For (iii) we need an additional definition: a vertex $v$ is said to be {\bf hot}  at some time $t$ if the proportion of its infected neighbors at time $t$ is larger than $\lambda/(8e)$. Then in \cite{CD} the authors proved (with different constants in the definition of lit and hot vertices, but this does not effect the proof) that 
\begin{itemize}
\item[$\bullet$] if $v$ is lit at some time $t$, then it becomes hot before $t+ \exp(c^* \lambda^2 \deg(v))$,
\item[$\bullet$]if $v$ is hot at some time $t$, then it remains lit until $t+ 2 \exp(c^* \lambda^2 \deg(v))$, 
\end{itemize}
with probability larger than $1- \exp(-c^* \lambda^2 \deg(v))$. Now (iii) follows from these claims.
\end{proof}
The following result is  Lemma 3.2 in \cite{MVY}.
\begin{lem} \label{lst} 
Let us consider the contact process on a graph $G=(V,E)$. There exist positive constants $c^*$ and $\lambda_0$, such that if $0 < \lambda <\lambda_0$, the following holds. Let $v $ and $ w$ be two  vertices satisfying  $\deg(v) \geq \frac{7}{c^*} \frac{1}{\lambda ^2} \log \left(\frac{1}{\lambda} \right) d(v,w)$. Assume that $v$ is lit at time $0$. Then $w$ is lit before $\exp(c^* \lambda^2 \deg(v))$ with probability larger than $1- 2 \exp (-c^* \lambda ^2 \deg(v))$.
\end{lem}
\section{Proof of  Proposition \ref{prp}}
In this section we  study the contact process on the P\'olya-point graph  $(T,o)$.  To prove Proposition 1.2, we have to show that 
\begin{align} \label{uopr}
\pp(\xi^o_t \neq \varnothing \,\,\forall t \geq 0) \lesssim \lambda^{1 + \frac{2}{\psi}} |\log \lambda|^{\frac{-1}{\psi}},
\end{align}
and 
\begin{align} \label{lopr}
\pp(\xi^o_t \neq \varnothing \,\,\forall t \geq 0) \gtrsim \lambda^{1 + \frac{2}{\psi}} |\log \lambda|^{\frac{-1}{\psi}}.
\end{align}
The proof of \eqref{uopr} is based on the proof of the upper bound in Proposition 1.4 in \cite{MVY} for the case of the contact process on Galton-Watson trees, and we put it in Appendix.

\subsection{Proof of \eqref{lopr}} In this part, we first estimate the probability  that there is an infinite sequence of vertices, including a neighbor of the root, with larger and larger degree and a small enough  distance between any two consecutive elements of the sequence.  We then repeatedly  apply Lemma \ref{lb} and \ref{lst}  to bound from below the probability that the virus  propagates along  these vertices, and like this survives forever. To this end, we denote by
 \begin{align} \label{csp}
 \varphi(\lambda) =\frac{7}{c^*} \frac{1}{\lambda ^2} \log \left(\frac{1}{\lambda} \right),
 \end{align}
 with $c^*$ as in Lemma \ref{lb} and \ref{lst}.
 
\noindent 
We denote by $ w_0 =(0)$, and $x_0 = x_{w_0}$.  For any $i \geq 1$, let 
$$w_i=(0,1,...,1) \textrm{ and } x_i = x_{w_i},$$
where $w_i$ has $i$ 1s. 
  Then $w_i$'s degree  conditioned on $x_i$  is $m+1$ plus a Poisson random variable with parameter
$$\frac{\gamma}{x_i^{\psi}} \int_{x_i}^1 \psi x^{\psi-1} dx = \gamma \frac{1- x_i^{\psi}}{x_i^{\psi}},$$
where $\gamma $ is a Gamma random variable with parameters $a=m+2mr+1$ and $1.$ Therefore letting $\kappa = (1-x_i^{\psi})/x_i^{\psi}$,  we have
\begin{align} \label{8}
 \mathbb{P}(\textrm{deg}(w_i)= m+1+k \mid x_i)& = \mathbb{E} \left( \frac{ e^{-\gamma \kappa}}{k!} (\gamma \kappa)^k \, \Big| x_i \right) =  \frac{\kappa ^k\Gamma(k+a)  }{ (\kappa+1)^{k+a} \Gamma(a) k!} \notag \\
& = \frac{\Gamma(k+a)}{ \Gamma(a) k!} (1-x_i^{\psi})^{k} x_i^{a \psi },
\end{align} 
where $\Gamma(b)= \int_0^{\infty} x^{b-1} e^{-x} dx$.
\begin{lem} \label{ehd}
 There is a positive constant $c$, such that for $\lambda$ small enough,
$$
\pnd (\mathcal{N}) \geq c \varphi(\lambda)^{-1/\psi},
$$
where 
 $$\mathcal{N}= \{ \exists (j_{\ell})_{\ell \geq 1}: j_{1}=1, \deg(w_{j_{\ell}}) \geq 2^{\ell+1} \varphi(\lambda)/ \psi \geq \varphi (\lambda) d(w_{j_{\ell}},w_{j_{\ell+1}}) \, \, \forall \ell \geq 1 \}.$$
\end{lem}
\begin{proof}
It follows from Markov's inequality  that for any $k\geq 1$,
 \begin{equation} \label{e29}
 \mathbb{P} \left( \prod_{i=1}^k  U_i > 2^{-(k+1)/2} \right) \leq 2^{-(k-1)/2},
 \end{equation}
 where $(U_i)$ is a sequence of i.i.d. uniform random variables  in $[0,1]$. 

 Since $\Gamma(k+a)/k! \asymp k^{a-1}$, there is a positive constant  $C$, such that for all $k \geq 1$
$$\frac{\Gamma(k+a)}{ \Gamma(a) k!}  \leq C k^{a-1}.$$ 
Then it follows from \eqref{8} that 
 \begin{equation} \label{e28}
 \pnd(\deg(w_i)\leq m+1+ (c/x_i^{ \psi}) \mid x_i) \leq \sum \limits_{k=0}^{[c /x_i^{ \psi}]} C k^{a-1}  x_i^{a \psi } \leq C c^{a},
 \end{equation}
 for any $c>0$. Set $j_1=1$ and $j_{\ell}=  [4\ell/ \psi]$ for $\ell \geq 2$. Then   define   
  $$\mathcal{N}_{\ell}=  \big \{  x_{j_\ell} \leq ( 4^{ \ell } \varphi(\lambda)/(c \psi))^{-1/\psi}, \textrm{deg}(w_{j_{\ell}}) \geq 2^{\ell+1} \varphi(\lambda)/\psi \big \} $$
  for all $\ell \geq 1$,  where $c$ is a positive constant to be chosen later. 
  
  \vspace{0,2 cm}
   Since $2^{\ell+1}\varphi(\lambda)/ \psi \geq 4 \varphi(\lambda)/ \psi =  \varphi(\lambda) d(w_{j_{\ell }}, w_{j_{\ell+1}})$,  we have
  \begin{equation} \label{deh}
  \kN \supset \bigcap _{\ell=1}^{\infty} \mathcal{N}_{\ell} .
  \end{equation}  
 Since $x_{j_{\ell}}$ is distributed as $x_1U_1 \ldots U_{j_{\ell}-1}$, applying $(\ref{e29})$ gives that
 $$
\pnd(x_{j_{\ell}} \leq x_1 4^{- \ell /\psi} ) \geq 1 - 2(4^{-\ell /\psi}).
 $$
 Therefore
 \begin{equation} \label{e30} 
 \pnd \left(x_{j_{\ell}} \leq  ( 4^{\ell} \varphi(\lambda)/(c \psi))^{-1/\psi} \mid \mathcal{N}_1 \right) \geq 1 - 2(4^{-\ell /\psi}).
 \end{equation}
  By using $(\ref{e28})$ with $(2^{1- \ell}c)$ instead of $c$ we obtain that 
 \begin{equation} \label{e31}
 \pnd \left(\textrm{deg}(w_{j_{\ell}} ) \geq 2^{\ell+1} \varphi(\lambda)/\psi \mid x_{j_{\ell}} \leq  ( 4^{\ell} \varphi(\lambda)/(c \psi))^{-1/\psi} \right) \geq 1- C (2^{1-\ell}c)^a.
  \end{equation}
Then  it follows from \eqref{e30} and \eqref{e31} that 
  \begin{equation} \label{e32}
  \pnd \left(\bigcap \limits_{\ell=2}^{\infty} \mathcal{N}_{\ell} \mid \mathcal{N}_1 \right) \geq 1- 2\sum \limits_{\ell=2}^{\infty}  4^{-\ell /\psi} - C\sum \limits_{\ell=2}^{\infty} (2^{1-\ell}c)^a \geq 1/4,
  \end{equation}
  provided $c$ is small enough.  We now estimate $\mathbb{P}(\mathcal{N}_1)$. Let 
  $$\bar{\varphi}(\lambda)=(4 \varphi(\lambda)/(c \psi ))^{-1/\psi}.$$  
 Recall that $x_1$ is  uniformly distributed on $[0,x_0]$, with $x_0 \sim U_0^{\chi}$ and $U_0 \sim \kU([0,1])$. Therefore
 \begin{align} \label{e33}
 \pnd (x_1 \leq \bar{\varphi}(\lambda) ) &= \mathbb{E} \left(\frac{\min\{\bar{\varphi}(\lambda) , x_0\}}{x_0} \right) \notag \\
 & \geq \bar{\varphi}(\lambda) \mathbb{P}(x_0 \geq \bar{\varphi}(\lambda)) \notag\\
 & = \bar{\varphi}(\lambda) (1-\bar{\varphi}(\lambda)^{1/\chi}) \notag \\
 &\geq \bar{\varphi}(\lambda)/2,
 \end{align}
  for $\lambda$ small enough. On the other hand,  \eqref{e28} gives that for $c$ small enough
  \begin{equation} \label{e34}
  \pnd(\mathcal{N}_1 \mid x_1 \leq \bar{\varphi}(\lambda) ) \geq 1- C c^{a} \geq 1/2.
  \end{equation}
 We thus can choose $c$ such that the two inequalities in \eqref{e32} and \eqref{e34} are satisfied.
Now it follows from \eqref{deh}, \eqref{e32}, \eqref{e33} and \eqref{e34} that 
$$\pnd (\kN) \gtrsim \bar{\varphi}(\lambda),$$
which implies the result.
\end{proof}
\noindent {\it Proof of \eqref{lopr}.} By repeatedly applying Lemma \ref{lst} to the pair of  vertices $(w_{i_{\ell}},w_{i_{\ell+1}})$,  we  obtain that 
\begin{align} \label{e36}
\mathbb{P}(\xi_t \neq \varnothing \, \forall t \geq T \mid \mathcal{N}, w_1 \textrm{ is lit at some time } T) &\geq 1 - 2 \sum \limits_{\ell =1}^{\infty} \exp(- c^* \lambda^2 2^{\ell+1} \varphi(\lambda)/\psi) \notag \\
& \geq 1- 2 \sum \limits_{\ell =1}^{\infty} \exp(- 7 (2^{\ell+1}) |\log  \lambda| / \psi) \notag \\
& \geq 1/2,
\end{align}
for $\lambda$ small enough. On the other hand, by using Lemma $\ref{lb}$ (i), we have
\begin{align} \label{e37}
&\mathbb{P}(w_1 \textrm{ is lit at some time } T \mid \kN, o \textrm{ is infected at  time } 0) \notag\\
& \geq c \lambda \mathbb{E}(1 - \exp(-c^* \lambda \deg(w_1)) \mid \kN ) \geq c \lambda/2,
\end{align} 
for some $c>0$ (note that on $\kN$, we have $c^* \lambda \deg(w_1) \geq 7$). Now it follows from \eqref{e36}, \eqref{e37} and Lemma \ref{ehd} that 
\[\mathbb{P}(\xi_t \neq \varnothing \, \forall t \geq 0) \geq (1/2) \times (c \lambda /2) \times \mathbb{P}(\kN) \gtrsim \lambda^{1 + 2/ \psi} |\log \lambda|^{-1/ \psi},\]
which proves \eqref{lopr}  \hfill $\square$

\section{Proof of Theorem $\ref{td}.$}
By using the self-duality of the contact process \eqref{sdl}, we see that to prove \eqref{g2}, it is sufficient to show that 
\begin{align} \label{tub}
\pn \left( \frac{1}{n} \sum \limits_{v \in V_n} 1(\{\xi^v_{t_n} \neq \varnothing \}) \leq C \lambda^{1+ 2/ \psi} |\log \lambda |^{-1/\psi}  \right) =1 -o(1),
\end{align} 
and 
\begin{align} \label{tlb}
\pn \left( \frac{1}{n} \sum \limits_{v \in V_n} 1(\{\xi^v_{t_n} \neq \varnothing \}) \geq c \lambda^{1+ 2/ \psi} |\log \lambda |^{-1/\psi}  \right) =1 -o(1),
\end{align} 
for some positive constants $c$ and $C$. We will prove these two statements in the next two subsections.
\subsection{Proof of \eqref{tub}} For $r \geq 1$, we define
\begin{align*}
\mathcal{L}_T(o,r) = \{(o,0) \leftrightarrow B_{T}(o,r)^c \times \mathbb{R}_+\},
\end{align*}
the event that the contact process on the P\'olya-point graph starting from the root infects vertices outside $B_{T}(o,r)$. Then we have 
\begin{align*}
\{\xi^o_t \neq \varnothing \,\forall t\} = \cap_{r=1}^{\infty} \mathcal{L}_T(o,r).
\end{align*}
Hence, it follows from \eqref{uopr} that there are positive constants $C$ and $R=R(\lambda)$, such that
\begin{align} \label{gobr}
\mathbb{P}(\mathcal{L}_T(o, R)) \leq C \lambda^{1 + \frac{2}{\psi}} |\log \lambda|^{\frac{-1}{\psi}}.
\end{align}
For any $v \in V_n$ and $R$ as in \eqref{gobr},  we define
\begin{align*}
\mathcal{L}_n(v,R) = \{(v,0) \leftrightarrow B_{G_n}(v,R)^c \times \mathbb{R}_+\}
\end{align*}
and 
$$X_v = 1 (\mathcal{L}_n(v,R)).$$
Theorem \ref{t2.3} yields that 
\begin{equation} \label{e23}
\lim \limits_{n \rightarrow \infty} \pn(\mathcal{L}_n(u,R) ) =  \mathbb{P}(\mathcal{L}_T(o, R)),
\end{equation}  
where $u$ is a uniformly chosen vertex from $V_n$. By combing this with \eqref{gobr} we obtain that
\begin{align*}
\lim \limits_{n \rightarrow \infty} \pn(X_u =1) \leq  C \lambda^{1+ 2/ \psi} |\log \lambda |^{-1/\psi},
\end{align*}
 or equivalently
\begin{align} \label{mx}
\lim \limits_{n \rightarrow \infty} \frac{1}{n} \sum \limits_{v \in V_n}\pn(X_v =1) \leq  C \lambda^{1+ 2/ \psi} |\log \lambda |^{-1/\psi}.
\end{align}
Now, let us consider  
$$W_n=\{(v,v') \in V_n \times V_n : d(v,v') \geq 2R+3\},$$
with $R$ as in \eqref{gobr}. Since $R +1 \leq b_2\log n / (\log \log n)$ for $n$ large enough,  Lemma \ref{ltyd} implies that
$$\sum \limits_{v,v' \in V_n}  \pn((v,v') \not \in W_n) = o(n^2).$$  
On the other hand, if $(v,v') \in W_n$ then $X_v$ and $X_{v'}$ are independent.  Therefore
\begin{align} \label{cx}
\sum \limits_{v,v' \in V_n} \cov(X_v, X_{v'} ) = o(n^2).
\end{align} 
Thanks to \eqref{mx} and \eqref{cx} by  using Chebyshev's inequality we get that 
\begin{align} \label{sx}
\pn \left( \frac{1}{n} \sum \limits_{v \in V_n} X_v \leq 2C \lambda^{1+ 2/ \psi} |\log \lambda |^{-1/\psi}  \right) =1 -o(1).
\end{align} 
Since the contact process on a finite ball in the P\'olya-point graph a.s. dies out, 
\begin{align*}
\lim \limits_{t \rightarrow \infty} \mathbb{P}(\kL_T(o,R)^c \cap \{\xi^{o}_t \neq \varnothing\}) =0.
\end{align*}
Hence for any $\varepsilon >0$, there exists $t_{\varepsilon}$, such that 
\begin{align} \label{vvq}
\mathbb{P}(\kL_T(o,R)^c \cap \{\xi^{o}_{t_{\varepsilon}} \neq \varnothing\}) \leq \varepsilon.
\end{align}
For any $v \in V_n$, define 
$$X_{v, \varepsilon} =1(\kL_n(v,R)^c \cap \{\xi^v_{t_{\varepsilon}} \neq \varnothing\}).$$
Then for $n$ large enough such that $t_n \geq t_{\varepsilon}$, we have 
\begin{align} \label{dtv}
1(\{\xi^v_{t_n} \neq \varnothing\}) \leq   X_v + X_{v, \varepsilon}.
\end{align}
It follows from Theorem \ref{t2.3} and \eqref{vvq}  that 
\begin{align*}
\lim \limits_{n \rightarrow \infty} \mathbb{E}_n(X_{u, \varepsilon}) = \mathbb{P}(\kL_T(o,R)^c \cap \{\xi^{o}_{t_{\varepsilon}} \neq \varnothing\}) \leq \varepsilon.
\end{align*}
By using this and Markov's inequality we get that for $n$ large enough, and for any $\eta >0$,
\begin{align} \label{bls}
\pn \left( \frac{1}{n} \sum \limits_{v \in V_n} X_{v, \varepsilon} > \eta  \right) \leq \frac{\sum _{v \in V_n}\en(X_{v, \varepsilon})}{n \eta} = \frac{\en(X_{u,\varepsilon})}{\eta} \leq \frac{2 \varepsilon}{\eta}.
\end{align}
By combining \eqref{sx}, \eqref{dtv} and \eqref{bls}, then letting $\varepsilon$ tend to $0$, we obtain that 
 \begin{align*} 
\pn \left( \frac{1}{n} \sum \limits_{v \in V_n} 1(\{\xi^v_{t_n} \neq \varnothing \}) \leq 3C \lambda^{1+ 2/ \psi} |\log \lambda |^{-1/\psi}  \right) =1 -o(1),
\end{align*}
which proves \eqref{tub}.  \hfill $\square$
\subsection{Proof of \eqref{tlb}}
This subsection is divided into three parts. In the first one, we will show that w.h.p. there are many vertices with large degree (larger than $\varkappa ^* \log n$). By using on the other hand that the diameter of the graph is  smaller than $b_1 \log n$, we can deduce that if one of these large degree vertices is infected, then the virus survives w.h.p.  for a time $\exp(c n/(\log n)^{1/ \psi})$, see Proposition \ref{pb}.  In the second part, we measure the  density of {\it potential} vertices which are promising for spreading the virus to some of these large degree vertices.  In the last part, we estimate the proportion of  potential vertices which really send the virus to large degree vertices, getting this way \eqref{tlb}.
\subsubsection{Lower bound on the extinction time}
Our aim in this part is to find large degree vertices as mentioned above. We then prove that if one of them is infected, the virus is likely to survive a long time. 
\begin{lem} \label{lm1} 
Let $\varkappa >0$ be given. Then there exists a positive constant $\bar{c} = \bar{c}(\varkappa)$, such that $\mathcal{A}_n$ holds w.h.p. with
 \begin{align*}
  \mathcal{A}_n= \{G_n\textrm{ contains $ \bar{c} n/(\log n)^{1/(1- \chi)}$  disjoint star graphs of size larger than } \varkappa \log n \}.
\end{align*} 
\end{lem}
\begin{proof}
 Let $\varepsilon  \in (0,1/3)$ be given, and let  us recall   the event $\kE_{\varepsilon}$ and the constant $K=K(\varepsilon)$ defined in  Lemma \ref{l2}: 
$$\mathcal{E}_{\varepsilon}= \Big \{ \Big| \frac{S_j}{S_k}  -  \left(\frac{j}{k}\right)^{\chi} \Big| \leq \varepsilon (j/k)^{\chi} \quad \forall \, K(\varepsilon) \leq j \leq  k \leq n\Big \}.$$ 
  Set $a_n = (M \log n)^{1/(1- \chi)}$, with $M$ to be chosen later. Denote by 
$$A=\{v_i : i  \in [n/a_n, 2n/a_n] \textrm{ and } \psi_i \in (\theta/i,  \mu/i)\}$$
and 
$$\mathcal{J}_0= \mathcal{E}_{\varepsilon} \cap \{|A| \geq \theta n / a_n\},$$
with $\theta$, $\mu$ as in Lemma \ref{l2}. Recall that the events $\{\psi_i \in (\theta/i,  \mu/i)\}$ are independent and have probability larger than $2\theta$. Therefore \eqref{ldv} implies that w.h.p. $|A| \geq \theta n / a_n$. Hence by Lemma \ref{l2} (ii), we have for $n$ large enough
$$\pn(\mathcal{J}_0) \geq 1- 2 \varepsilon.$$
We now suppose that $\mathcal{J}_0$ occurs. Then,  the elements of $A$ can be written   as $ \{ v_{j_1},...,v_{j_{\ell}}\}$ with $ \ell \in [ \theta n/a_n, n/a_n]$.
Then define
$$A_1= \{v_j : n/2\leq j \leq n\} .$$
We will show that all vertices in $A$ have a large number of neighbors in $A_1$.   Indeed,  it follows from \eqref{ecp} and Lemma \ref{l2} (ii) that for $K(\varepsilon) \leq j < k$,
\[\pn(v_j \sim v_k \mid (\varphi_t), \kE_{\varepsilon}) \asymp \frac{\psi_j S_j}{S_{k-1}} \asymp \psi_j \left(\frac{j }{k-1}\right)^{\varkappa}  .\]
Hence,  for all $v_j \in A$ and $v_k \in A_1$,
\begin{equation} \label{eq5}
 \pn(v_j \sim v_k \mid \kJ_0) \asymp   \frac{ a_n ^{1- \chi}}{n}.
\end{equation}
Conditionally on $(\psi_j)$,  the events $\{\{v_{j_1} \sim v_k\}\}_{k \in A_1}$ are independent. Hence thanks to \eqref{ldv} we get that there are positive constants $\theta_1$, $c,$ and $ C$,  such that
\begin{align*}
\pn \left(  c a_n ^{1- \chi}\leq \sum \limits_{v_k \in A_1} 1(v_{j_1} \sim v_{k } ) \leq  C a_n ^{1- \chi} \mid \kE_{\varepsilon} \right) \geq 1- \exp(-  \theta_1  a_n ^{1- \chi}),
\end{align*}
or equivalently
\begin{equation} \label{eq6}
\pn(\mathcal{J}_1 \mid \mathcal{J}_0 ) \geq 1 - \exp(-  \theta_1  a_n ^{1- \chi}), 
\end{equation} 
where 
$$\mathcal{J}_1 = \{c a_n ^{1- \chi} \leq |B_1| \leq C a_n ^{1- \chi}\}$$
and 
$$B_1 = \{v_k \in A_1 :  v_{j_1} \sim v_k\}.$$
Note that in this proof, the value of the constant $\theta_1$  may change from line to line. Now let us consider  
$$A_2 = A_1 \setminus B_1 \quad \textrm{ and } \quad B_2 = \{v_k \in A_2 :  v_{j_2} \sim v_k\}.$$
We notice that on $\mathcal{J}_1 \cap \mathcal{E}$, the cardinality of $A_2 $ is larger than $  n/2 - C a_n ^{1- \chi} \geq n/4.$ Thus, similarly to \eqref{eq6}, there exist positive constants $c_1 $ and $C_1 $, such that 
\begin{equation} \label{j21}
\pn (\mathcal{J}_2 \mid \mathcal{J}_1 \cap \mathcal{J}_0 ) \geq 1 - \exp(- \theta_1  a_n ^{1- \chi}),
\end{equation}
where 
$$\mathcal{J}_2 = \{c_1 a_n ^{1- \chi}  \leq |B_2| \leq C_1 a_n ^{1- \chi}\}.$$
Here we can also assume that $c_1 \leq c$ and $C_1 \geq C$. Likewise for all $2\leq s \leq \ell$, define recursively
$$A_s = A_{s-1} \setminus B_{s-1}, \quad B_s = \{v_k \in A_s : v_{j_s} \sim v_k\},$$
$$\mathcal{J}_s = \{c_1 a_n ^{1- \chi}\leq |B_s| \leq C_1 a_n ^{1- \chi}\}.$$
On $ \underset{i=0}{ \overset{s-1}{\cap}}  \mathcal{J}_i$, we have $|A_s| \geq n/2 - s C_1 a_n ^{1- \chi} \geq n/4$. Therefore, similarly to \eqref{j21} 
$$\pn \left(\mathcal{J}_s \mid  \underset{i=0}{ \overset{s-1}{\cap}}  \mathcal{J}_i \right) \geq 1- \exp- (\theta_1  a_n ^{1- \chi}).$$
Hence 
$$\pn \left(\underset{i=1}{ \overset{\ell}{\cap}}  \mathcal{J}_i \mid \mathcal{J}_0 \right) \geq 1 -  n \exp (- \theta_1  a_n ^{1- \chi})/a_n .$$
Taking $M$ large enough such that $c_1 a_n ^{1- \chi} \geq \varkappa \log n$ and $n \exp (- \theta_1  a_n ^{1- \chi}) \leq 1$ yields that
\begin{align} \label{cbs}
\pn ( |B_s| \geq \varkappa \log n \, \, \forall \, 1 \leq s \leq \ell  \mid \mathcal{J}_0) \geq 1 -a_n^{-1}. 
\end{align} 
Moreover, by  definition  $B_s \cap B_t = \varnothing $ for all $s \neq t$. Hence, all vertices in $A$ have more than $\varkappa \log n$ distinct neighbors. Finally, take $\bar{c}$  such that $\bar{c}n/(\log n)^{1/(1- \chi)} \leq \theta n/ a_n$, for instance $\bar{c} \leq \theta M^{-1/(1- \chi)}$. In conclusion, we have shown that for any given $\varepsilon \in (0,1/3)$,
$$\pn(\mathcal{A}_n) \geq 1- 2\varepsilon - a_n^{-1} \geq 1- 3 \varepsilon,$$
for $n$ large enough. Since this holds for any $\varepsilon >0$, the result follows.
\end{proof}
\noindent   To determine the constant $\varkappa$ in the definition of $\kA_n$, we first recall that  $$\pn(\mathcal{B}_n) = 1-o(1),$$
where
 $$\mathcal{B}_n= \{d(G_n) \leq b_1 \log n\}.$$
 Hence to apply Lemma \ref{lst} to the large degree vertices exhibited in the previous lemma, we need
$$\varkappa \log n \geq \frac{7}{c^*} \frac{1}{\lambda ^2} \log \left(\frac{1}{\lambda} \right) b_1 \log n.$$
Moreover, in \eqref{cslk}, we  will use that $\varkappa \geq 3/(c^* \lambda^2)$. So we let
\begin{equation} \label{c}
 \varkappa^* = \max \Big \{\frac{7}{ c^* } \frac{1}{\lambda ^2} \log \left(\frac{1}{\lambda} \right) b_1, \frac{3}{ c^* \lambda^2} \Big \}.
\end{equation}
 Then we let $\bar{c}^*= \bar{c}^*(\varkappa^*) $ and $\kA_n$ be defined accordingly as in Lemma \ref{lm1}.

\vspace{0,2 cm}
 A set of vertices $V = \{w_1,\ldots, w_k\} \subset V_n$ is called {\bf good} if $|S(w_i) \setminus \cup_{j\neq i} S(w_j)| \geq \varkappa ^* \log n$ for all $1 \leq i \leq k$, where $S(v)$ denotes the star graph formed by $v$ and its neighbors.
 
   Let  $V^*_n$ be  a maximal good set i.e. $|V^*_n| = \max \{|V|: V \subset V_n \textrm{ is good}\}$. 
    \begin{prop} \label{pb}  There exists a positive constant $c$, such that 
$$\pn \left(\xi_{T_n} \neq \varnothing  \mid \xi_0 \cap V_n^{*} \neq \varnothing  \right) = 1- o(1),$$ 
where $T_n=\exp(c \lambda^2 n/(\log n)^{1/\psi} )$. 
\end{prop}
\begin{proof}
Thanks to Lemma \ref{lmdi} and \ref{lm1}, we can assume that $d(G_n) \leq b_1 \log n$ and 
 $|V^*_n| \geq \bar{c}^* n/ (\log n)^{1/(1-\chi)}$.  Assume also that at time $0$ a vertex in $V^*_n$, say $v$,  is infected.

 Due to the definition of $V^*_n$,  for any $ w \in V^*_n$, we can select from the set of $w$'s neighbors a subset $D(w)$ of size $\varkappa^* \log n$, such that $D(w) \cap D(w')=\varnothing$ for all $w \neq w'$. 

\noindent We say that a vertex $w$ in $V^*_n$ is {\it infested} at some time $t$ if the proportion of infected sites in $D(w)$ at time $t$ is larger than  $\lambda/(16e)$ (the term is taken from \cite{MMVY}).  

  It follows from Lemma \ref{lb} (ii) that $v$ becomes infested  with probability tending to $1$, as $n \rightarrow \infty$. Using  Lemma \ref{lb} (iii) and \ref{lst} (note that $|D(w)| \geq (7/(c^* \lambda^2))|\log \lambda| d(w,w')$), we deduce that for any $t \geq 0$ and $w \in V_n^*$,
  \begin{align*}
  &\pn( w \textrm{ is infested at } t+  2 \exp(c^* \lambda^2 \varkappa^* \log n) \mid v \textrm{ is infested at }t) \\ & \geq  1- 4 \exp (-c^* \lambda ^2 \varkappa^* \log n ).
  \end{align*}
Therefore  
\begin{align}
& \pn(  \textrm{All vertices in $V^*_n$ are infested at } t+  2 \exp(c^* \lambda^2 \varkappa^* \log n) \mid v \textrm{ is infested at }t) \notag \\ & \geq  1- 4n \exp (-c^* \lambda ^2 \varkappa^* \log n )  \notag\\
& \geq 1-n^{-1}, \label{cslk}
\end{align}
where  we have used that $c^* \lambda ^2 \varkappa^*  \geq 3$.  Now if all vertices in $V^*_n$ are infested at the same time, then  the proof of Proposition 1 in \cite{CD} shows that the virus survives a time exponential in  $\sum _{v \in V^*_n} \deg(v)$. More precisely, let $I_{n,t}$ be the number of infested vertices in $V^*_n$ at time $t$. Then there is a positive constant $\eta$, such that for all $k \leq |V^*_n|$,
$$\pn \left(I_{n, s_k} \geq k/2 \mid I_{n,0} \geq k   \right) \geq 1- s_k^{-1},$$
where $s_k = \exp(\eta \lambda^2 k \varkappa^* \log n)$. The result  follows by taking $k= \big[\bar{c}^* n/(\log n)^{1/(1-\chi)} \big]$.
\end{proof}
\subsubsection{Density of potential vertices} In this part we will estimate the proportion of {\it potential} sites from where the virus can be sent with positive probability to a vertex at distance quite small (of order $(\log \log n)^2$) and with large degree (larger than $ \varkappa^* \log n$). 

This proportion approximates the probability that there is an infection path from the uniformly chosen vertex $u$ to a vertex with degree larger than $ \varkappa ^* \log  n$.  To bound from below this probability, we use the same ideas as in  Lemma \ref{ehd}. Indeed, we will find a sequence of vertices starting from a neighbor of $u$ and ending at a large degree vertex, satisfying the hypothesis of Lemma \ref{lst} for spreading the virus from $u$ to the ending vertex (see Lemma \ref{ldeh}).

Here are just some comments on the order of magnitude above. First, if a vertex with degree larger than $\varkappa^* \log  n$ is infected, then w.h.p. it will infect a site in $V^*_n$ , and  then  we can conclude with Proposition $\ref{pb}$.  Secondly,  $(\log \log n)^2$  is the distance from a potential vertex to a large degree  vertex  and is  much smaller than the typical distance between two different potential vertices. Hence the propagation of the virus from these potential vertices to their closest large degree vertex are approximately independent events.

\vspace{0,2 cm}
Set
$$R_n= \left[(\log \log n)^2 \right].$$
For  $w \in V_n$,  define $k_0(w)$  by $w= v_{k_0(w)}$, and for $i\geq 1$  let  $k_i(w)$ be chosen arbitrarily from the indices such that  $v_{k_i(w)} $ receives an edge emanating from $ v_{k_{i-1}(w)}$.  We define also 
 $$\mathcal{H}_n (w)=\{ k_0 (w) \geq n/\log n\} \cap \{  k_{i+1} (w) \geq k_i(w) / \log k_i (w)  \geq n^{1/2} \, \,\forall \, 0 \leq i \leq R_n\}. $$
\begin{lem} \label{lbn}
 There is  a positive constant  $\theta_0$, such that for all $\theta \leq \theta_0$, for all $\varepsilon  \in (0, 1/2)$, and for any vertex $w$,  we have
 \begin{itemize}
 \item[(i)] for all $i \leq R_n$
 \[ \pn \left( \max \limits_{ v \in B_{G_n}(w,i)} \deg(v) \geq \theta e^{\theta i} (n/k_{0}(w))^{1-\chi} \mid \kE_{\varepsilon} \cap \kH_n(w)\right) \geq 1 - e^{-\theta i},\]
 \item[(ii)] $ \pn \left( \kH_n(w) \mid  k_0(w) \geq n/ \log n, \kE_{\varepsilon}\right) =1 -o(1/ \log \log n),$
 \end{itemize}
 with $\kE_{\varepsilon}$ as in Lemma \ref{l2} (ii).
\end{lem}
\begin{proof} We first prove (ii).  It follows from the construction of  $G_n$ and Lemma \ref{l2} (ii) that if $k_i(w)/ \log k_i(w) \geq K(\varepsilon)$, then  
\begin{align*}
\pn(k_{i+1} (w) \leq k_i (w)/ \log k_i (w) \mid \mathcal{E}_{\varepsilon}, k_i(w), (\varphi_t)) = \frac{S_{[k_i(w)/ \log k_i(w)]}}{S_{k_i(w)-1}} \leq (1+\varepsilon) \left(\frac{1 }{\log k_i(w)}\right)^{\chi} . 
\end{align*}
Hence  for all $i \leq R_n$,
\begin{eqnarray*}
 & &\pn \left(k_{i+1} (w)  \geq n/(\log n)^{i+2} \mid \mathcal{E}_{\varepsilon}, k_i (w) \geq n/ (\log n)^{i+1} \right) \\
  & \geq & \pn \left(k_{i+1} (w)  \geq k_i(w)/ \log k_i(w)  \mid \mathcal{E}_{\varepsilon}, k_i (w) \geq n/ (\log n)^{i+1} \right) \\
  & = & 1- o((\log n)^{-\chi /2}).
\end{eqnarray*}
Therefore
\begin{eqnarray*}
 & &\pn \left(k_{i+1} (w)  \geq k_i(w)/ \log k_i(w) \geq n/(\log n)^{i+2} \,\, \forall \, i \leq R_n \mid \mathcal{E}_{\varepsilon}, k_0 (w) \geq n/ \log n \right) \\
  & = & 1-  o(R_n (\log n)^{-\chi /2}) =1 - o(1/ \log \log n).
\end{eqnarray*}
This implies (ii), since for all $i\leq R_n$
\[n/(\log n)^{i+2} \geq \sqrt{n}.\]
We now prove (i). First, we claim that there is a positive constant $c_0$, such that for any $c<c_0$, there exists $c'=c'(c)>0$, such that for all $i \leq R_n$
\begin{itemize}
\item[(a)] $\pn\left(k_{[i/2]} (w) \leq e^{-ci} k_0(w) \mid \mathcal{E}_{\varepsilon} \cap \mathcal{H}_n(w) \right) \geq 1 - e^{- c'i}$, \\
\item[(b)] $\pn\left( \exists j \in (i/2, i): \psi_{k_j(w)} \geq c/ k_j(w) \mid \mathcal{E}_{\varepsilon} \cap \mathcal{H}_n(w) \right) \geq 1 - e^{- c'i}$, \\
\item[(c)] $\pn(\deg(v_k) \geq c' (n/k)^{1- \chi} \mid \psi_{k} \geq c/ k, \mathcal{E}_{\varepsilon}  ) \geq 1 - \exp(- c' (n/k)^{1- \chi})$,  for any $v_k \in V_n$. 
\\
\end{itemize}
 From these claims  we can deduce the result. Indeed, (a) and (b) imply that with probability larger than $1- 2 \exp(-c' i)$ there is an integer $j \in (i/2, i)$ such that $k_{j} (w) \leq e^{-ci} k_0(w)$ and $\psi_{k_j(w)} \geq c/ k_j(w)$. Then (c) gives that $\deg(v_{k_j(w)}) \geq c' e^{c(1-\chi) i}(n/k_0(w))^{1- \chi}$ with probability larger than $1- \exp(-c'e^{c(1-\chi) i})$. Hence (ii)  follows by taking $\theta$ small enough.

To prove (a), similarly to Lemma \ref{ltyd}, we consider 
 $$X_j(w)=1(\{k_j(w) \leq k_{j-1}(w)/2\}) \textrm{ and } \kF_j(w) = \sigma (k_t(w): t \leq j) \vee \sigma((\varphi_t)).$$
On $\kH_n(w)$, we have $ K(\varepsilon) \leq \sqrt{n} \leq k_j (w)$ for all $j \leq R_n$. Then by using the same argument as in Lemma \ref{ltyd} we obtain that on $\kH_n (w)\cap \kE_{\varepsilon}$,
 \begin{align} \label{xjw}
 \en(X_j(w) \mid \kF_{j-1}(w)) \geq \frac{S_{[k_j(w)/2]} - S_{[k_j(w)/\log k_j(w)]}}{S_{k_j(w)-1}} \geq p
 \end{align}
 and 
  \begin{align} \label{tdeh}
 \pn \left( \sum \limits_{j=1}^{[i/2]}  X_j (w)\geq ip/4 \right) \geq 1 - 2\exp(-ip^2/16),
 \end{align}
 for some constant $p>0$. Since  $k_{[i/2]}(w) \leq 2^{-ip/4} k_0(w)$ as soon as $\sum _{j=1}^{[i/2]}  X_j(w) \geq ip/4$, the claim (a)  follows from \eqref{tdeh}. 
 
 We now prove (b). Let $\theta$ be the constant as in Lemma \ref{l2} (v). Fix some $j \in (i/2, i)$ and set    $k=k_j(w)-1$ and $ \ell = [k_j(w)/ \log k_j(w)]$. Then we have 
 \begin{eqnarray} \label{kjj}
 \pn \left(  \psi_{k_{j+1}(w)} \geq \theta/ k_{j+1}(w) \mid k, \ell \right) &=& \en \left( \frac{1}{S_{k}-S_{\ell}} \sum \limits_{t=\ell+1}^{k} \varphi_t 1( \psi_t \geq \theta/t)  \mid k, \ell \right) \notag \\
 & \geq & \en \left( \frac{1}{S_{k}} \sum \limits_{t=\ell+1}^{k} \varphi_t 1( \psi_t \geq \theta/t)  \mid k, \ell \right)
 \end{eqnarray}
  On $\kH_n(w) \cap \kE_{\varepsilon}$, we have $k \geq k_{j+1}(w) \geq \ell \geq \sqrt{n} \geq K(\varepsilon)$. Hence, using Lemma \ref{l2} (ii) and the fact that $S_n=1$,  we get  on $\kH_n(w) \cap \kE_{\varepsilon}$,
  \begin{eqnarray} \label{skch}
  |S_k - (k/n)^{\chi}| \leq \varepsilon (k/n)^{\chi} \quad \textrm{and} \quad |S_{\ell} - (\ell/n)^{\chi}| \leq \varepsilon (\ell/n)^{\chi}.
  \end{eqnarray}
 Moreover, by Lemma \ref{l2} (v) for all $t$, we have
 \begin{equation} \label{ppsgt}
 \en[\varphi_t 1(\psi_t \geq \theta/t)] \geq \theta.
 \end{equation}
 Now using \eqref{skch} and \eqref{ppsgt}, we obtain that for all $n$ large enough 
 \begin{eqnarray} \label{yhl}
\en \left( \frac{1}{S_{k}} \sum \limits_{t=\ell+1}^{k} \varphi_t 1( \psi_t \geq \theta/t)  \mid k, \ell, \kH_n(w) \cap \kE_{\varepsilon} \right) &\geq & \frac{n^{\chi}}{(1+ \varepsilon) k^{\chi} } \en \left(  \theta \sum \limits_{t=\ell+1}^{k} \varphi_t \mid k, \ell, \kH_n(w) \cap \kE_{\varepsilon} \right) \notag\\
&\geq & \frac{\theta n^{\chi}  }{(1+ \varepsilon) k^{\chi} }   \en \left(  S_k - S_{\ell}\mid k, \ell, \kH_n(w) \cap \kE_{\varepsilon} \right)  \notag\\
& = & \frac{\theta n^{\chi}}{(1+ \varepsilon) k^{\chi} }  \left[(1-\varepsilon) \left(k/n \right) ^{\chi} - (1+\varepsilon) \left(\ell/n \right) ^{\chi}  \right]  \notag\\
& \geq & \theta/4,
\end{eqnarray}
since $\varepsilon \in (0,1/2)$ and $\ell =[(k+1)/ \log (k+1)]$. It follows from \eqref{kjj} and \eqref{yhl} that 
 \begin{eqnarray} \label{pgttt} 
 \pn \left(  \psi_{k_{j+1}(w)} \geq \theta/ k_{j+1}(w) \mid k_j(w),\kH_n(w) \cap \kE_{\varepsilon} \right) \geq \theta/4. 
 \end{eqnarray}
 Now it follows from \eqref{pgttt} that 
\begin{equation*}
\pn \left(  \nexists j \in (i/2,i): \psi_{k_{j+1}(w)} \geq \theta/ k_{j+1}(w) \mid \kH_n(w) \cap \kE_{\varepsilon} \right) \leq (1-\theta/4)^{[i/2]}, 
\end{equation*} 
which implies (b).  Finally, (c) can be proved as   \eqref{eq6}.
\end{proof}

\begin{lem} \label{ldeh} Let u be a uniformly chosen vertex from $V_n$. Then 
$$
\pn( \mathcal{M}) \gtrsim \lambda \varphi(\lambda) ^{-1/\psi} ,
$$
where
\begin{align*}
 \mathcal{M} =  \{  \exists w  \in B_{G_n} (u,R_n) : \deg(w) \geq \varkappa^*\log n \} \cap \{ (\xi^u_t) \textrm{ makes $w$ lit inside } B_{G_n} (u,R_n) \}.
\end{align*}
\end{lem}
\begin{proof}
 Define $k_0$ by $v_{k_0}= u$ and for $i \geq 1$ let   $k_i$  be chosen arbitrarily from the indices that  $v_{k_i} $ receives an edge emanating from $v_{k_{i-1}}$. Let us denote $u_1=v_{k_1}$ and define also 
 $$\mathcal{H}_n := \kH_n(u_1)= \{ k_1 \geq n/\log n\} \cap \{  k_{i+1} \geq k_i / \log k_i  \geq \sqrt{n}  \, \,\forall \, 1\leq i \leq R_n +1 \}. $$
In this proof, we assume that $\varepsilon = o(\lambda \varphi(\lambda) ^{-1/\psi})$ as $\lambda \rightarrow 0$. Similarly to  Lemma \ref{lbn} by using that $k_0$ is chosen uniformly from $\{1, \ldots, n\}$, we have $\pn (\kH_n \mid \kE_{\varepsilon}) =1 -o(1/ \log \log n)$ and hence   $\pn(\mathcal{E}_{\varepsilon} \cap \kH_n)=1-o(\lambda \varphi(\lambda) ^{-1/\psi})$. We assume now that these two events happen.

We recall the claim (c) in the proof of Lemma \ref{lbn}: there is a positive constant $c_0$, such that for any $c<c_0$, there exists $c'=c'(c)>0$, such that  
\begin{align} \label{mdqr}
\pn(\deg(v_k) \geq c' (n/k)^{1- \chi} \mid \psi_{k} \geq c/ k) \geq 1 - \exp(- c' (n/k)^{1- \chi}),
\end{align}
for any $v_k \in V_n$.  Let us consider
$$\mathcal{M}_1 =\big\{ k_1 \leq n/\tilde{\varphi}(\lambda) \big\}, $$
where 
$$\tilde{\varphi}(\lambda) =( 4\varphi(\lambda)/ c'\theta^2)^{1/1- \chi},$$
 with $\theta$ a small enough constant (smaller than $\theta_0$ as in Lemma \ref{l2} and \ref{lbn} and than $c_0$), and $c' =c'(\theta)$.  Define
$$\kM_2 = \kM_1 \cap \big \{ \forall \, 1 \leq  \ell \leq R_n' \,\,\, \exists w_{\ell} : d(u_1, w_{\ell}) \leq r_{\ell},  \deg (w_{\ell})\geq \varphi(\lambda) \exp( \theta r_{\ell})  \big \},$$
where  $$r_{\ell}= 4 \ell/ \theta^2  \quad \textrm{for} \quad 1 \leq \ell \leq R_n' := \theta^2 R_n/8.$$
 By applying Lemma \ref{lbn} for $u_1$  we get that for any $\ell \leq R_n'$
\begin{align*}
\pn \left( \max \limits_{ v \in B_{G_n}(u_1, r_{\ell})} \deg(v) \geq \theta e^{\theta r_{\ell}} (n/k_{1})^{1-\chi}  \right) \geq 1 - e^{-\theta r_{\ell}}.
\end{align*}
If $k_{1} \leq n/ \tilde{\varphi} (\lambda)$, then  $$ \theta \exp(\theta r_{\ell}) (n/k_{1})^{1-\chi} \geq \theta \exp(\theta r_{\ell}) \tilde{\varphi} (\lambda)^{1- \chi} \geq  \varphi(\lambda) \exp( \theta r_{\ell}).$$
 Thus 
\begin{align*}
\pn \left( \exists v : d(v, u_1) \leq r_{\ell}, \deg(v) \geq \varphi(\lambda) \exp( \theta r_{\ell}) \mid \kM_1 \right) \geq 1 - e^{-4\ell/\theta}.
\end{align*}
Hence  
\begin{align} \label{m21}
\pn\left( \kM_2 \mid \kM_1\right) & \geq 1- \sum \limits_{\ell =1}^{R_n'} \exp(-4 \ell/ \theta) \geq 1- 2 \exp(-4/ \theta). 
\end{align}
Define 
$$\kM_3 = \kM_1 \cap \{\deg(u_1) \geq 4 \varphi(\lambda)/ \theta^2\}.$$
Similarly to \eqref{yhl}, we can show that
\begin{align*} 
 \pn \left(  \psi_{k_1} \geq \theta/ k_1 \mid  k_1 \leq n/ \tilde{\varphi}(\lambda)\right)  \geq \theta/4.
\end{align*}
Using \eqref{mdqr} and the fact that $c' \tilde{\varphi}(\lambda)^{1- \chi} = 4 \varphi(\lambda)/ \theta^2$,  we get 
\begin{align*}
\pn \left( \deg(u_1) \geq c' (n/k_1)^{1- \chi} \mid k_1  \leq n/ \tilde{\varphi} (\lambda), \psi_{k_1} \geq \theta/ k_1 \right) \geq 1- \exp(-4\varphi(\lambda)/ \theta^2 ) \geq 1/2.
\end{align*}
From the last two inequalities we deduce that 
\begin{align*}
\pn (\kM_3 \mid \kM_1) \geq \theta/8.
\end{align*}
Combining this with \eqref{m21} we obtain that
\begin{align} \label{m231}
\pn (\kM_2 \cap \kM_3 \mid \kM_1) \geq \theta/8 - 2 \exp(-4/\theta) \geq \theta/16.
\end{align} 
We now bound from below $\pn(\mathcal{M}_1).$  Observe that
\begin{align*}
 \pn\left( k_1 \leq n/ \tilde{\varphi} (\lambda) \, \Big |\, k_0 , (\varphi_j)  \right)  & \geq   \frac{S_{[n/ \tilde{\varphi}(\lambda)]} 1(k_0 >n/\tilde{\varphi}(\lambda))}{S_{k_0-1}} \\
& \gtrsim \tilde{\varphi} (\lambda) ^{- \chi} \left( \frac{k_0}{n}\right)^{\chi} 1(k_0 > n/\tilde{\varphi}(\lambda)).
\end{align*}
Since $k_0$ is distributed uniformly on $\{1, \ldots ,n\}$, we get 
\begin{align*}
\en [ (k_0/n)^{\chi} 1(k_0 >n/\tilde{\varphi}(\lambda)) ]  \asymp 1.
\end{align*}
Therefore 
\begin{align*} 
\pn\left( \kM_1  \right) \gtrsim \tilde{\varphi} (\lambda)^{-\chi} \gtrsim \varphi(\lambda) ^{-1/\psi}.
\end{align*}
This and \eqref{m231} give that
\begin{align} \label{m23}
\pn\left( \kM_2 \cap \kM_3  \right) \gtrsim \varphi(\lambda) ^{-1/\psi}.
\end{align}
Observe that on $\kM_2 \cap \kM_3$, we have $\deg(u_1) \geq \varphi(\lambda) r_1 \geq \varphi(\lambda) d(u_1,w_1) $ and 
$$\deg(w_{\ell}) \geq \varphi(\lambda) \exp( \theta r_{\ell}) \geq  2 \varphi(\lambda) r_{\ell+1} \geq  \varphi(\lambda) d(w_{\ell}, w_{\ell+1})$$ 
for all $1 \leq \ell \leq R_n'$. In other words,  $u_1$ and the vertices $(w_{\ell})$ satisfy the condition in Lemma \ref{lst}, and thus applying  this lemma inductively  yields that 
\begin{align} \label{xm} 
& \pn( \textrm{$w_{R_n'} $ is lit inside } B_{G_n}( u_1,R_n) \mid \kM_2 \cap \kM_3, u_1 \textrm{ is lit}) \notag \\
&\geq 1- \sum \limits_{\ell =1}^{R_n'} \exp(- c^* \lambda^2 \varphi(\lambda )e^{\theta r_{\ell}} ) \notag \\
& \gtrsim 1.
\end{align}
Similarly to \eqref{e37},  the probability that  $(\xi^u_t)$ makes $u_1$  lit  is of order $\lambda$. It follows from this and  \eqref{xm}  that 
 \begin{align} \label{xm1} 
\pn( (\xi^{u}_t) \textrm{ makes $w_{R_n'} $ lit inside } B_{G_n}(u, R_n) \mid \kM_2 \cap \kM_3) \gtrsim \lambda.
\end{align}
In addition,   $\deg(w_{R_n'}) \geq \varkappa ^* \log n$. Therefore  
\begin{align*}
\kM \supset \kM_2 \cap \kM_3 \cap \{(\xi^{u}_t) \textrm{ makes $w_{R_n'} $ lit inside } B_{G_n}(u,R_n)\}.
\end{align*}
 Combining this  with \eqref{m23} and \eqref{xm1} gives the result.
\end{proof}
\subsubsection{Proof of \eqref{tlb}} For any $v \in V_n$, we define 
\begin{align*}
 Y_v = 1( \{  \exists w  \in B_{G_n} (v,R_n): \deg(w) \geq \varkappa^* \log n \} \cap \{ (\xi^v_.) \textrm{ makes $w$ lit inside } B_{G_n} (v,R_n) \})
\end{align*}
and 
$$Z_v = Y_v 1(\{\xi^v_{T_n} \neq \varnothing\}),$$
where $T_n$ is as in Proposition \ref{pb}. Then 
\begin{equation} \label{znhx}
\sum \limits_{v\in V_n} Z_v \leq \sum \limits_{v\in V_n} 1(\{\xi^v_{T_n} \neq \varnothing\}).
\end{equation}
\begin{lem} \label{ls}The following assertions hold:
\begin{itemize}
\item[(i)] 
$\pn \left( \frac{1}{n} \sum \limits_{v \in V_n} Y_v \geq c \lambda \varphi(\lambda)^{-1/\psi}  \right) =1 -o(1)$, for some $c>0$, independent of $\lambda$. 
\item[(ii)] $\pn \left(  Z_v=1 \mid Y_v=1   \right) \rightarrow 1$, as $n \rightarrow \infty $ uniformly in $v \in V_n$. 
\end{itemize}
\end{lem}
\begin{proof}
For (i), let $\varepsilon \in (0, 1/2)$ be given. We have to show that the probability in the left-hand side is larger than $1- 2\varepsilon$ for $n$ large enough. First, Lemma \ref{ldeh} implies that
\begin{align*}
\pn( Y_u =1) \gtrsim \lambda \varphi(\lambda)^{-1/\psi},
\end{align*}
or equivalently
\begin{align*} 
\frac{1}{n} \sum \limits_{v\in V_n} \pn(Y_v=1) \gtrsim \lambda\varphi(\lambda)^{-1/\psi}.
\end{align*}
Using Chebyshev's inequality, the result follows from  this and the following claim: on $\kE_{\varepsilon}$
\begin{align} \label{tcy}
\sum \limits_{v,v' \in V_n} \cov(Y_v, Y_{v'} ) =o(n^2).
\end{align}
To prove it, we consider
\begin{align*}
\kV_n =\{(v_i,v_j): i,j \geq n/\log n, d(v_i,v_j) \geq 2R_n+3 \}.
\end{align*}
Since $R_n +1 \leq b_2 \log n/ (\log \log n)$ for $n$ large enough, it follows from Lemma \ref{ltyd}  that 
\begin{align} \label{ccv}
\sum \limits_{v,v'  \in V_n} \pn ((v,v') \not \in \kV_n) =o(n^2).
\end{align}
On the other hand, Lemma \ref{lbn} gives that if $i \geq n/ \log n$, then 
\begin{align*}
 \pn \left( \exists w  \in B_{G_n} (v_i,R_n) : \deg(w) \geq \varkappa^* \log n \mid \kE_{\varepsilon} \right) =1-o(1/ \log \log n).
\end{align*}
Moreover, given the graph $G_n$, $Y_v$ and $Y_{v'}$  only depend on the Poisson processes defined on the vertices and edges on the balls $B_{G_n}(v,R_n) $ and $ B_{G_n}(v',R_n)$ respectively. Hence on $\kE_{\varepsilon}$ for all $(v,v') \in \kV_n$,  
\begin{align} \label{cy}
\cov(Y_v, Y_{v'} ) =o(1/ \log \log n).
\end{align}
Now \eqref{tcy} follows from \eqref{ccv} and \eqref{cy}.

We now prove (ii).  If $Y_v =1$, then there exists a vertex $w$ such that  $\deg(w) \geq \varkappa^* \log n$ and $w$ is lit at some time. Besides, on $\kB_n$ the diameter of the graph is bounded by $b_1 \log n$ w.h.p. Hence similarly to Lemma \ref{lst}, we can show that on $ \mathcal{B}_n$
\begin{align*}
\pn( w \textrm{ infects a vertex in $V^*_n$ }   )  & \geq 1-  \exp(- c^* \varkappa^* \lambda^2\log n). 
\end{align*}
If one of the vertices in $V^*_n$ is infected, it follows from Proposition \ref{pb} that w.h.p.  the virus survives up to time $T_n$. Hence we obtain (ii) by using that  $\kB_n$ holds w.h.p.
  \end{proof}
It  follows from  \eqref{znhx}, Lemma \ref{ls} and  Markov's inequality that w.h.p.
\begin{equation*}
\frac{1}{n} \sum \limits_{v\in V_n} 1(\{\xi^v_{T_n} \neq \varnothing\}) \gtrsim  \lambda \varphi(\lambda)^{-1/\psi},
\end{equation*}
which proves \eqref{tlb}. \hfill $\square$
\section{Proof of Proposition \ref{cvelpa}}
Let us first recall a result which we will use below.   
\begin{prop} \cite[Lemma A.1]{MMVY} \label{pM} 
Let $(G_n) = (V_n, E_n)$ be a sequence of graphs, such that $|G_n|= n$, for all $n$. Let $(\xi_t)_{t\geq 0}$ be the contact process on $G_n$ starting from full occupancy and let $\tau_n$ be the extinction time of this process.  Assume that there exist sequences of positive numbers $(a_n), (b_n)$ satisfying
\begin{itemize}
\item[(i)] $\lim \limits_{n \rightarrow \infty} a_n =\lim \limits_{n \rightarrow \infty} b_n = \infty, \lim \limits_{n \rightarrow \infty} \frac{a_n}{b_n} =0$;
\item[(ii)] $ \lim \limits_{n \rightarrow \infty} \sup \limits_{A \subset V_n}  \pn (\xi^A_{a_n} \neq \xi_{a_n}, \xi_{a_n} \neq \varnothing) = 0$;
\item[(iii)] $\lim \limits_{n \rightarrow \infty}\pn (\tau_n < b_n) = 0$.
\end{itemize}
Then, $\tau_n/ \en (\tau_n)$ converges in distribution, as $n \rightarrow \infty$, to the exponential distribution of parameter $1$.
\end{prop}
This result for the case of finite boxes in $\Z^d$ is proved by Mountford in \cite{M}. Then it is stated for general case as above. We will apply this result to prove the following proposition \ref{cpcel} which  implies  Proposition \ref{cvelpa}.  
\begin{prop} \label{cpcel}
Let $(G_n)$ be a sequence of connected graphs, such that $|G_n|= n$, for all $n$. 
Let $\tau_n$ denote the extinction time of the contact process on $G_n$ 
starting from full occupancy. Assume that
\begin{align} \label{cnas}
\frac{D_{n,\max}}{d_n \vee \log n} \rightarrow \infty,
\end{align}
with $D_{n,\max}$ the maximum degree and $d_n$ the diameter of $G_n$.  Then 
\begin{align*} 
\frac{\tau_n}{\en (\tau_n)}\quad  \mathop{\longrightarrow}^{(\kL)}_{n\to \infty} \quad  \kE(1).
\end{align*}
\end{prop} 
\begin{proof} First we set $\bar{\lambda}= \lambda \wedge 1$. Using Lemma \ref{lb}, we get  w.h.p.
\begin{align}
\label{ctaunDnmax}
\tau_n \geq \exp(c \bar{\lambda}^2 D_{n,\max}), 
\end{align}
with $c=c^*/2$ and $c^*$ as Lemma \ref{lb}.

Hence, according to Proposition \ref{pM},   it suffices to show that there exists a sequence $(a_n)$, such that
\begin{equation} \label{apt}
a_n = o (\exp(c \bar{\lambda}^2 D_{n,\max}))
\end{equation}
and
\begin{eqnarray}
\label{cxivxi}
\sup_{v\in V_n}\, \pn(\xi^v_{a_n} \neq \xi_{a_n}, \xi^v_{a_n} \neq \varnothing) = o(1),
\end{eqnarray}
where $(\xi_t)_{t\ge 0}$ denotes the process starting from full occupancy.

Using  \eqref{cnas}, we can find a sequence $(k_n)$ tending to infinity, such that
\begin{align}
\label{cDnmaxdn}
\frac{D_{n,\max}}{(\log n \vee d_n)k_n} \rightarrow \infty.
\end{align}
Now define 
\begin{align} \label{avbn}
b_n= \exp(c \bar{\lambda}^2 (\log n \vee d_n) k_n ) \quad \textrm{and}\quad a_n=4b_n+1.  
\end{align}
Then \eqref{ctaunDnmax} and \eqref{cDnmaxdn} show that $a_n$ satisfies \eqref{apt}, so it amounts now to prove \eqref{cxivxi} for this choice of $(a_n)$. To this end it is convenient to introduce the dual contact process. 
Given some positive real $t$ and 
$A$ a subset of the vertex set $V_n$ of $G_n$,  the dual process $(\hat{\xi}^{A,t}_s)_{s\le t}$ is defined by 
\[\hat{\xi}^{A,t}_s = \{ v\in V_n : (v,t-s)\longleftrightarrow A \times \{ t \} \},\]
for all $s\le t$. It follows from the graphical construction that for any $v$, 
\begin{eqnarray} \label{ccl2}
&&\nonumber \pn(\xi^v_{a_n} \neq \xi_{a_n}, \xi^v_{a_n} \neq \varnothing)\\
 &=& \pn (\exists w\in V_n : \xi^v_{a_n}(w) = 0,\, \xi^v_{a_n} \neq \varnothing,\, \hat{\xi}^{w,a_n}_{a_n} \neq \varnothing) \notag\\
&\le & \sum_{w\in V_n} \pn\left(\xi^v_{a_n} \neq \varnothing,\, \hat{\xi}^{w,a_n}_{a_n} \neq \varnothing, \textrm{ and } \hat{\xi}^{w,a_n}_{a_n-t} \cap  \xi^v_t  = \varnothing \textrm{ for all } t\le a_n\right), 
\end{eqnarray}
So let us prove now that the last sum above tends to $0$ when $n\to \infty$. Set 
$$\beta_n = [k_n (d_n \vee \log n)],$$
and let $u$ be a vertex with degree larger than $\beta_n$. Let then $S(u)$ be a star graph of size $\beta_n$ centered at $u$.  
As in Proposition \ref{pb}, we say that $u$ is infested if the number of its infected neighbors \textit{in $S(u)$} is larger than $\bar{ \lambda} \beta_n/(16e)$. 
We first claim that 
\begin{align}
\pp(\xi^v_{b_n} \neq \varnothing, u \textrm{ is  not infested before } b_n )  = o(1/n). \label{cvbn1}
\end{align}  
To see this, define $K_n=[b_n/d_n]$ and for any $0 \leq k \leq K_n-1$
\[A_k:=\{\xi^v_{kd_n}\neq \varnothing\},\] 
and 
\[B_k:= \left\{ \xi_{kd_n}^v\times\{kd_n\} \longleftrightarrow (u,(k+1)d_n-1) \right\} \cap \{u \textrm{ is infested at time } (k+1)d_n\}.\]
 Note that 
\begin{align}
\label{cinc.bn}
\{\xi^v_{b_n} \neq \varnothing, u \textrm{ is  not infested  before } b_n\} \ \subset \  \bigcap_{k=0}^{K_n-1} A_k \cap B_k^c.
\end{align} 
It is not difficult to see that  
\begin{eqnarray*}
\pn\left((z,t)\longleftrightarrow (z',t+d_n-1)\right) \ge \exp(-C d_n) \quad \textrm{for any $z, z'\in V_n$ and $t\ge 0$},
\end{eqnarray*}
for some constant $C>0$.  On the other hand,  Lemma \ref{lb} (i) implies that  if $u$ is infected at time $t$ then it is infested at time $t+1$ with probability larger than $1/3$, if $n$ is large enough.
Therefore for any $k\le K_n-1$, 
$$\pn(B_k^c\mid \kG_k){\bf 1}(A_k) \le 1-\exp(-Cd_n)/3,$$
with $\kG_k$ the sigma-field generated by  the contact process up to time $ kd_n$. Iterating this, we get  
\begin{eqnarray*}
\pn\left(\bigcap_{k=0}^{K_n-1} A_k \cap B_k^c\right) &\le  & (1-\exp(-Cd_n)/3)^{K_n-1} = o(1/n),
\end{eqnarray*}
where the last equality follows from the definition of $b_n$. Together with \eqref{cinc.bn} this proves our claim \eqref{cvbn1}. 
Then by using Lemma \ref{lb} (iii) we get 
\begin{align} \label{cvb}
\pp(\xi^v_{b_n} \neq \varnothing, u \textrm{ is not infested at time } 2b_n ) = o(1/n).
\end{align}
Therefore, if we define 
\begin{align*}
\kA(v)&=\{\xi^v_{b_n} \neq \varnothing, u \textrm{ is infested at time $2b_n$}  \}, 
\end{align*}
we get 
$$
\pp(\kA(v)^c, \xi^v_{b_n} \neq \varnothing)=o(1/n).
$$
Likewise if 
\begin{align*}
\hat{\kA}(w)&= \{\hat{\xi}^{w,4b_n+1}_{b_n}  \neq  \varnothing,  \exists \, U \subset S(u): |U| \geq \frac{\bar{\lambda}}{16e}\beta_n\textrm{ and } (x,2b_n+1) \leftrightarrow (w,4b_n+1) \, \forall \, x \in U \}.
\end{align*}
then  
$$ 
\pp(\hat{\kA}(w)^c, \hat{\xi}^{w, 4b_n+1}_{b_n} \neq \varnothing)=o(1/n). 
$$
Moreover, $\kA(v)$ and $\hat{\kA}(w)$ are independent for all $v$, $w$. 
Now the result will follow if we can show that for any $A,B \subset S(u)$ with $|A|, |B|$ larger than $\bar{\lambda}\beta_n/(16e)$
\begin{eqnarray} \label{caub}
\pp(A \times \{2b_n\} \mathop{\longleftrightarrow}^{S(u)} B \times \{2b_n+1\} ) = 1-o(1/n),
\end{eqnarray} 
where the notation 
$$A \times \{2b_n\} \mathop{\longleftrightarrow}^{S(u)} B \times \{2b_n+1\}$$ 
means that there is an infection path inside $S(u)$ from a vertex in $A$ at time $2b_n$ to a vertex in $B$ at time $2b_n+1$. 
To prove \eqref{caub}, define
\begin{align*}
\bar{A} &=\{x \in A \setminus \{u\}: \bar{A}_{x} \textrm{ holds } \},\\
\bar{B} &=\{y \in B \setminus \{u\}: \bar{A}_{y} \textrm{ holds }\},
\end{align*}
where 
\[\bar{A}_{x}  = \{\textrm{ there is no recovery at $x$ between $2b_n$ and $2b_n+1$}\}.\]
Observe that   
$$\pp(\bar{A}_{x} \textrm{ holds})= 1-e^{-1}.$$
Therefore,  the standard large deviations results show that $|\bar{A}|$ and $|\bar B|$ are larger than $(1-e^{-1}) \bar{\lambda} \beta_n/(32e)$, with probability at least $1-o(1/n)$. Now let 
$$\kE= \{|\bar{A}| \geq (1-e^{-1}) \bar{\lambda} \beta_n/(32e)\} \cap \{|\bar{B}| \geq (1-e^{-1}) \bar{\lambda} \beta_n/(32e)\}.
$$
We set  
\begin{align*}
\varepsilon_n = \frac{1}{(\log n) \sqrt{k_n}} \quad \textrm{and} \quad J_n = \left[ \frac{(\log n) \sqrt{k_n}}{2}\right],
\end{align*}
with $(k_n)$ as in \eqref{cDnmaxdn}, and define for $0 \leq j \leq J_n-1$
\begin{eqnarray*}
C_j& =& \{\textrm{ there is no recovery at $u$ between $2b_n+2j \varepsilon_n$ and $2b_n+ (2j +2) \varepsilon _n$}\} \\
  & \cap & \{\exists x \in \bar{A}: \textrm{ there is an infection from $x$ to $u$   between $2b_n+2j \varepsilon_n$ and $2b_n+ (2j +1) \varepsilon _n$}\} \\
& \cap & \{\exists y \in \bar{B}: \textrm{ there is an infection from $u$ to $y$   between $2b_n+(2j+1) \varepsilon_n$ and $2b_n+ (2j +2) \varepsilon _n$}\}.
\end{eqnarray*}
Observe that 
\begin{eqnarray}
\label{cunionCj}
\bigcup_{j=0}^{J_n-1} C_j \subset \Big\{A \times \{2b_n\} \mathop{\longleftrightarrow}^{S(u)} B \times \{2b_n+1\}\Big\}. 
\end{eqnarray} 
Moreover, conditionally on $\bar A$ and $\bar B$, the events $(C_j)$ are independent, and
\begin{align*}
\pp(C_j \mid \bar A, \bar B) & = e^{-2 \varepsilon_n} \pp(\bin(|\bar A|, 1- e^{-\varepsilon_n}) \geq 1)\times \pp(\bin(|\bar B|, 1- e^{-\varepsilon_n}) \geq 1)\\
&\ge  1/2,
\end{align*}
on the event $\kE$, if $n$ is large enough. Therefore
\begin{eqnarray*}
\pn\left(\kE, \,  \bigcap_{j=0}^{J_n-1} C_j^c\right) &\le  & 2^{-J_n} = o(1/n).
\end{eqnarray*}
This together with \eqref{cunionCj} imply  \eqref{caub}, and concludes the proof of the proposition.
\end{proof}

{\it Proof of Proposition \ref{cvelpa}}.  Using the same arguments for \eqref{eq6}, we have for any sequence $(a_n)$ tending to infinity
\begin{equation*}
\pp \Big(\exists \ell \in [c n/ a_n, n/a_n]: \deg(v_{ \ell }) \geq c a_n^{1 - \chi} \Big) =1-o(1),
\end{equation*}
for some $c>0$. By taking $a_n = \sqrt{n}$, we obtain that 
\begin{equation} \label{xdgs}
\pp (D_{n,\max} \geq cn^{(1- \chi)/2}) = 1-o(1).
\end{equation}
Now Proposition \ref{cvelpa}  follows from Proposition \ref{cpcel}, Lemma \ref{lmdi} and \eqref{xdgs}. \hfill $\square$  
\section{Appendix:  Proof of \eqref{uopr}}
  We closely follow the proof given in \cite{MVY} for Galton-Watson trees.  However, we have to take care that in our situation the degrees of the vertices are not independent as in Galton-Watson trees. This leads to some complications.

\vspace{0.2 cm}
To simplify the computation, we consider a modified version of the P\'olya-point graph defined as follows: we use the same construction  except that now $m_v = m$ and $\gamma_v \sim F'$ for all vertices. Then, the new tree stochastically dominates the original tree (note that $F \preceq F'$) and thus, it is sufficient to prove the upper bound for the contact process on this new tree. In this appendix, we always consider this modified graph, and for simplicity, we use the same notation as for the P\'olya-point graph. Now our goal is to prove that 
\begin{align} \label{muopr}
\pp(\xi^o_t \neq \varnothing \,\,\forall t \geq 0) \lesssim \lambda^{1 + \frac{2}{\psi}} |\log \lambda|^{\frac{-1}{\psi}},
\end{align}
where $(\xi^o_t)_{t\geq 0}$ is the contact process on the modified P\'olya-point graph starting from the root $o$.

\vspace{0.2 cm}
To describe more precisely the distribution of the P\'olya-point graph, we recall a basic fact of Poisson  processes (see for example  \cite[Chapter 2]{DV}).

\vspace{0.2 cm}
\noindent {\bf Claim}. For any $a<b$, let $\Lambda (a,b)$ be the set of arrivals of the Poisson  process on $[a,b]$ with intensity $f(x)$.  Then conditional  on $|\Lambda (a,b)|=k$, these $k$ arrivals are independently  distributed on $[a,b]$ with the same density $f(x)/ \int_a^b f(t) dt$.

\vspace{0.2 cm}

 We observe that for any vertex $v$ in the P\'olya-point graph, conditioned on its position $x_v$ and  its number of descendants $m+k$,   $x_{(v,1)}, \ldots, x_{(v,m)}$   are i.i.d. uniform random variables on $[0, x_v]$, and  $x_{(v,m+1)}, \ldots, x_{(v,m+k)}$ are arrivals of a Poisson process on  $[x_v,1]$ with intensity $\gamma_v \frac{\psi x^{\psi-1}}{1-x_v^{\psi}} dx$ (conditional on having $k$ arrivals). Therefore  $x_{(v,m+1)}, \ldots, x_{(v,m+k)}$ are independently  distributed on $[x_v,1]$ with  the same density $\frac{\psi x^{\psi-1}}{1-x_v^{\psi}} dx$.

 On the other hand,   similarly to \eqref{8}, for all $v$ with $a=m+2mr+1$,
\begin{align} \label{pkx}
p(k,x) & = \pp(\textrm{number of descendants of } v =m+k \mid x_v=x) \notag \\
 & = \frac{\Gamma(k+a)}{ \Gamma(a) k!} (1-x^{\psi})^{k} x^{a \psi } \notag \\
 & \asymp k^{a-1} (1-x^{\psi})^k x^{a \psi},
\end{align} 
since $\Gamma(k+a)/k! \asymp k^{a-1}$. Moreover,
\begin{align*}
\sum \limits_{k \leq M} k^{a-1} (1-x^{\psi})^k & \asymp \sum \limits_{k =0 }^{ M \wedge [x^{-\psi}]} k^{a-1} + \int_{M \wedge x^{-\psi}}^M  \exp(-t x^{\psi}) t^{a-1} dt \\
& \asymp (M \wedge x^{-\psi} )^{a} + x^{-a\psi}\int_{M x^{\psi} \wedge 1}^{ M x^{\psi}}  \exp(-u ) u^{a-1} du \\
& \asymp (M \wedge x^{-\psi} )^{a}.
\end{align*}
Therefore
\begin{align} \label{tpk0}
\sum \limits_{k \leq M} p(k,x)\asymp  (M \wedge x^{-\psi} )^{a} x^{a \psi}.
\end{align}
Similarly,
\begin{align} 
\sum \limits_{k \leq M} k p(k,x) & \asymp  (M \wedge x^{-\psi} )^{a+1} x^{a \psi}, \label{tpk1} \\
\sum \limits_{k \leq M} k^2 p(k,x) & \asymp  (M \wedge x^{-\psi} )^{a+2} x^{a \psi}, \label{tpk2} 
\end{align}
Using \eqref{tpk1} and taking $M$ tend to infinity, we have
\begin{equation} \label{tkpk}
\sum \limits_{k\geq 0} k p(k,x)  = \kO(  x^{-\psi}). 
\end{equation}
 Let $r >0$ and $ M \in \mathbb{N}$ be given. For any vertex $v$, define the truncated tree starting from $v$ as
\begin{align*}
T_{r,M}^v=\{v\} \cup & \{ w  \textrm{ descendant of } v :   d(v,w) \leq r, \deg(y) \leq M \\ 
 & \textrm{ for all }   y \not \in \{v,w\}  \textrm { in the geodesic from $v$ to } w \},
\end{align*}
and for $1 \leq i \leq r$, set
\begin{align*}
T^{v}_{i,r,M} & = \{ w  \in  T^v_{r,M} : d(v,w)=i  \} .\\
S^{v}_{i,r,M} &= \{w  : d(v,w)=i \textrm{  and $w$ is a leaf of } T^v_{r,M} \} .
\end{align*}
If $v=o$, we simply write $T_{r,m}$, $T_{i,r,m}$, and $S_{i,r,m}$. By definition, if $v \neq o$, then
\begin{align}
|T^{v}_{1,r,M}| & = \deg(v) -1, \label{t1v} \\
|S^{v}_{1,r,M}| & = |\{w  : w  \textrm{ is a child of $v$ with }  \deg(w) >M \}|, \label{s1v} \\
|T^{v}_{i+1,r,M}| &= \sum \limits_{j=1}^{\deg(v)-1} |T^{(v,j)}_{i,r,M}| 1( \deg((v,j))\leq M) \quad \textrm {for }  \, 1 \leq i \leq r-1, \label{tiv}\\
|S^{v}_{i+1,r,M}| &= \sum \limits_{j=1}^{\deg(v)-1} |S^{(v,j)}_{i,r,M}| 1( \deg((v,j))\leq M) \quad \textrm {for }  \, 1 \leq i \leq r-2, \label{siv} \\
|S^{v}_{r,r,M}| &= |T^{v}_{r,r,M}|. \label{trv}
\end{align}
If $v=o$, we just replace $\deg(o)-1$ by $\deg(o)$ in these equations.

As in  \cite{MVY} (more precisely, Sections 6.1, 6.2 and 6.3), we can see that the proof of \eqref{muopr}  follows from the four following lemmas. 
\begin{lem} \label{lup1}
There is a positive constant $C$, such that for all $1 \leq i \leq r$ and  $M\geq m+1$,
\begin{align*}
\E(|T_{i,r,M}|) \leq C^i(\log M)^{i-1}. 
\end{align*}
Furthermore for all $1 \leq i \leq r-1$ and $M\geq m+1$,
\begin{align*}
\E(|S_{i,r,M}|) \leq C ^i(\log M)^{i-1} M^{-1/ \psi}. 
\end{align*}
\end{lem}
In fact, the bound for $\E(|T_{i,r,M}|)$ plays the same role as the estimate (6.2) in \cite{MVY}, and the  bound for $\E(|S_{i,r,M}|)$ is similar to  (6.1) in \cite{MVY}.
\begin{lem}  \label{lup2}
For all $r >0$ and $M\geq m+1$, we have
\begin{align*} 
\E(|S_{1,r,M}|1(|S_{1,r,M}| \geq 2) \mid \deg(o) \leq M) = \kO(M^{-1-1/\psi} \log M).
\end{align*}
\end{lem}
This result is used in the estimate in (6.7) in \cite{MVY}.
\begin{lem} \label{lup3}
For all $r >0$ and $ M' \geq  M \geq m+1$, we have
$$\pp( \deg(o) \leq M, |S_{1,r,M}|=1)= \kO(M^{-1/\psi}),$$
$$\pp(\deg(o^*) \geq M' \mid \deg(o) \leq M, S_{1,r,M}=\{o^*\})= \kO\left( \left(M/M'\right)^{1/\psi}\right).$$
\end{lem}
The first estimate is used for the event $ B^4_4$ in \cite{MVY}, and the second one is an analogue of the bound for $q[M', \infty)/q[M, \infty)$ in Proposition 6.3  (at the first line of page 27).

For any  $r >0$ and $M' \geq  M \geq m+1$, we define the  conditional probability measures 
\begin{align*}
 \mathbb{Q}_1 (\cdot) & = \pp(\cdot \mid \deg(o) \leq M, |S_{1,r,M}| = 1),\\
 \mathbb{Q}_2 (\cdot) & = \pp(\cdot \mid \deg(o) \leq M, S_{1,r,M} = \{o^*\}, \deg(o^*) = M').
\end{align*}
We call $T^*$ the tree  $T$ rooted 	at $o^*$. Similarly  as for $T$, we also define $T^*_{i,r,M}$, $S^*_{i,r,M}$.
\begin{lem} \label{lup4}
There is a positive constant $C$, such that for all $1 \leq i \leq r$ and $M' \geq M \geq m+1$
\begin{align*}
 \E_{\mathbb{Q}_1}(|T_{i,r,M}|) \leq (C \log M)^{i}. 
\end{align*}
\begin{align*}
 \E_{\mathbb{Q}_2}(|T^*_{i,r,M}|) \leq C^i( \log M)^{i-1} M'. 
\end{align*}
Furthermore for all $1 \leq i \leq r-1$ and $M \geq m+1$,
\begin{align*}
 \E_{\mathbb{Q}_2}(|S^*_{i,r,M}|) \leq C^i(\log M)^{i-1} M'M^{-1}.
\end{align*}
\end{lem} 
The bound for $\E_{\mathbb{Q}_1}(|T_{i,r,M}|)$ is used in (6.12) in \cite{MVY},  the bounds for $\E_{\mathbb{Q}_2}(|T^*_{i,r,M}|)$ and $\E_{\mathbb{Q}_2}(|S^*_{i,r,M}|)$ are used in their Section 6.3.

\vspace{0.3 cm} 
Assume that Lemmas \ref{lup1}--\ref{lup4} hold for a moment, we now prove the upper bound of the survival probability of the contact process on $T$.

\vspace{0.3 cm}

\noindent {\it Proof of \eqref{muopr}}. Using  the same notation in \cite{MVY}, we set  
\begin{align*}
M = \left \lceil \frac{1}{8 \lambda^2} \right \rceil \textrm{ and } R = \left \lceil \frac{2/\psi+5}{2/\psi-1} \right \rceil  +1.
\end{align*}
and define
\begin{eqnarray*}
B_1^4 &= & \{ \deg(o) >M \},\\
B_2^4 &= & \left \{ \deg(o) \leq M, (o,0) \leftrightarrow \left(\bigcup _{i=2} ^{R} S_{i,R,M} \right) \times \mathbb{R}_+ \textrm{ in } T_{R,M} \right\},\\
B^4_3 &= &\left \{ \deg(o) \leq M, |S_{1,R,M}| \geq 2, (o,0) \leftrightarrow  S_{1,R,M}  \times \mathbb{R}_+ \textrm{ in } T_{R,M} \right\}. 
\end{eqnarray*} 
When $S_{1,R,M} = \{o^* \}$, let  $0<t^*<t^{**}$ be  the first two arrival times of the process $D_{(o, o^*)}$ (the Poisson process of intensity $\lambda$ representing the infections from $o$ to $o^*$).  Then, we define 
$$B^4_4 = \{  \deg(o) \leq M, |S_{1,R,M}| =1, t^{**} < \inf D_o \}.$$
We say that $o^*$ becomes infected {\it directly} if $t^* < \inf D_o$. We say that  it becomes infected {\it indirectly} if there are  infection paths from $o$ to $o^*$ but all these paths must visit at least one vertex different from $o$ and $o^*$ . Define  
$$ B^4_5 =  \{\deg(o) \leq M, |S_{1,R,M}| =1, \exists \,y \in T_{R,M}, 0 < s< t: (o,0) \leftrightarrow (y,s) \leftrightarrow (o,t) \textrm{ inside } T_{R,M } \}.$$
Note that if $o^*$ becomes infected indirectly, then $B^4_5$ must occur. Let us define 
\begin{align*} 
B^4_6=  \{ \deg(o) \leq M,|S_{1,R,M}|=1, t^* < \inf D_o,  (o^*, t^ *) \leftrightarrow B(o,R)^c \times [t^*, \infty) \}.
 \end{align*}
Then
\begin{equation} \label{potr}
\{ (o,0) \leftrightarrow B(o,R)^c \times \mathbb{R}_+ \} \subset \bigcup_{i=1}^6 B_i^4.
\end{equation}
{\bf Event} $B^4_1$. We observe that by \eqref{pkx}  
\begin{eqnarray*}
\pp(\deg(o)>M \mid x_o =x) & \asymp & x^{a \psi}\sum \limits_{k>M} k^{a-1} (1- x^{\psi})^k \notag \\
& \lesssim & y^{a \psi} \int_M^{\infty} t^{a-1} \exp(-t x^{\psi}) dt \notag \\
&  \asymp &\int_{Mx^{\psi}}^{\infty} t^{a-1} \exp(-t ) dt  \notag \\
& \lesssim &   \exp(-M x^{\psi}/2). 
\end{eqnarray*}
By combining this with the fact that $x_o$ is distributed on $[0,1]$ with density $(\psi+1) x^{\psi} dx$, we get
\begin{eqnarray*}
\mathbb{P} (B^4_1) &=& \int_0^1 \pp(\deg(o)>M \mid x_o =x) (\psi+1) x^{\psi} dx \\
& \lesssim &  \int_0^1 \exp(-M x^{\psi}/2) x^{\psi} dx \\
& \lesssim & M^{-1-1/\psi} = \kO(\lambda^{2 +2/\psi}).
\end{eqnarray*} 
{\bf Event} $ B_2^4$. As in \cite{MVY}, we have 
\begin{eqnarray*}
\mathbb{P} (B^4_2)  &\leq & \sum _{i=2} ^R (2 \lambda) ^i \mathbb{E} (|S_{i,R,M}|) \\
 & \leq &   M^{-1/ \psi} \sum _{i=2} ^{R-1}  (C \lambda \log M )^i + (C \lambda \log M )^R \\
& \lesssim & \lambda^{-2/ \psi  } \sum _{i=2} ^{R-1}  (\lambda |\log \lambda|)^i + (\lambda |\log \lambda|)^R \\
& \lesssim  & \lambda^{3/2 + 2/ \psi }. 
\end{eqnarray*}
Here, for the first inequality we have used Lemma \ref{lup1} and \eqref{trv}. 

\vspace{0.2 cm}
\noindent {\bf Event} $ B^4_3$. As in \cite{MVY}, we have for $\lambda$ small enough
\begin{eqnarray*}
\mathbb{P}(B^4_3) &\leq & (2 \lambda )  \mathbb{E}( |S_{1,R,M}| 1 (|S_{1,R,M}| \geq 2  ) \mid  \deg(o) \leq M ) \\
& =& \kO( \lambda M^{-1 -1/ \psi} \log M) = \kO( \lambda \times \lambda^{2 +2/ \psi} |\log \lambda|)\\
&= & \kO(\lambda^{2 + 2/ \psi}),
\end{eqnarray*}
where we have used Lemma \ref{lup2} in the second line. 

\vspace{0.2 cm}
\noindent {\bf Event} $B^4_4$. The number of transmissions from $o$ to $o^*$ before  time $t$ has Poisson distribution with parameter $\lambda t$. Thus, 
\begin{eqnarray*} \mathbb{P}(B^4_4) & = &\mathbb{P}( \deg(o) \leq M, |S_{1,R,M}| = 1 )  \int_0^{\infty} \mathbb{P}( \textrm{Poi}( \lambda t) \geq 2 ) e^{-t} dt \\
 & =& \kO(\lambda^{2/ \psi})  \int_0^{\infty} \lambda^2 t^2  e^{-t} dt \\
 &=& \kO(\lambda^{2+ 2/ \psi}).
 \end{eqnarray*}
Note that in the second line, we have used Lemma \ref{lup3} to estimate the first term and the fact that $ \mathbb{P}( \textrm{Poi}(u) \geq 2 ) \leq u^2$ to bound the second one.

\vspace{0.2 cm}
\noindent {\bf Event} $B^4_5$. As in \cite{MVY},
\begin{displaymath}
\mathbb{P}(B^4_5) \leq \mathbb{P} ( \deg(o) \leq M, |S_{1,R,M}|= 1 ) \mathbb{P}\left( 
\begin{array}{c|c}
\exists \, y \in T_{R,M}, 0<s<t: & \deg(o) \leq M,\\
(o,0) \leftrightarrow (y,s)\leftrightarrow (o,t)   \textrm{ inside } T_{R,M} & |S_{1,R,M}| =1
\end{array} \right).
\end{displaymath}
By Lemma \ref{lup3}, the first term is   $\kO( \lambda ^{2/\psi}$). On the other hand, by the same argument in \cite{MVY} the second one is bounded by  
\begin{eqnarray*} 
& &\sum_{i=1}^{R} \lambda^{2i} \times \mathbb{E}(|\{x \in T_{R,M}: d(o,x)=i \}| \mid \deg(o) \leq M, |S_{1,R,M}| =1) \\
& =& \sum_{i=1}^{R} \lambda^{2i} \times \mathbb{E}_{\Q_1}(|T_{i,R,M}|) \\
& \leq & \sum_{i=1}^{R} ( C \lambda^2 |\log \lambda|)^{i} \\
& =& \kO(\lambda^{3/2}).
\end{eqnarray*} 
Thus we have  
$$\mathbb{P}(B^4_5) = \kO( \lambda ^{3/2 +2/ \psi }).$$

\noindent {\bf Event} $B^4_6$. We have
\begin{eqnarray*}
\mathbb{P}(B^4_6)& \leq & \mathbb{P}(\deg(o) \leq M, |S^1_{R,M}|=1) \times \mathbb{P}(t^* < \inf D_o)  \\
&  & \times   \mathbb{P}((o^*, t^*) \leftrightarrow B(o,R)^c \times [t^*, \infty ) \mid \deg(o) \leq M, |S^1_{R,M}|=1, t^* < \inf D_o) \\
& \lesssim &   \lambda^{2/ \psi } \times \lambda \times   \mathbb{P} ((o^*, 0) \leftrightarrow  B(o,R)^c \times \mathbb{R}_+ \mid \deg(o) \leq M, |S^1_{R,M}|=1). 
\end{eqnarray*}
Note that we have used Lemma \ref{lup3} to bound the first probability. Now, it remains to bound the third term. We observe that   for any $M' >M$,
\begin{eqnarray*}
& & \mathbb{P}((o^*, 0) \leftrightarrow B(o,R)^c \times \mathbb{R}_+ \mid  \deg(o) \leq M, |S_{1,R,M}|=1) \\
& = &\sum \limits_{k=M+1}^{\infty}  \mathbb{P}((o^*, 0) \leftrightarrow B(o,R)^c \times \mathbb{R}_+ \mid \deg(o^*) =k, \deg(o) \leq M, S_{1,R,M}=\{o^*\} )   \\
 & & \hspace{1cm} \times   \mathbb{P}(\deg(o^*)=k \mid \deg(o) \leq M, S_{1,R,M}=\{o^*\}  )\\
&\leq & \mathbb{P}((o^*, 0) \leftrightarrow B(o,R)^c \times \mathbb{R}_+ \mid \deg(o^*) =M', \deg(o) \leq M, S_{1,R,M}=\{o^*\} ) \\
& &+ \mathbb{P}(\deg(o^*) \geq M' \mid \deg(o) \leq M, S_{1,R,M}=\{o^*\}  ) \\
& \lesssim & \Q_2((o^*, 0) \leftrightarrow B(o,R)^c \times \mathbb{R}_+) + (M/M')^{1/\psi}.
\end{eqnarray*}
Here, we used Lemma \ref{lup3} to bound the second term and $\Q_2$ is the conditional probability depending on $M'$ which was defined in Lemma \ref{lup4}.

As in \cite{MVY},  we take 
\[M'= \lceil \varepsilon_1 \lambda^{-2} |\log \lambda| \rceil,\]
with 
\[\varepsilon_1= \varepsilon_1'/64 \qquad \textrm{and} \qquad \varepsilon_1'= \min \{(2/\psi-1), 2 \}/4. \]
Then  the second term is of order  
\[(M/M')^{1/\psi} \asymp |\log \lambda|^{-1/\psi}.\]
To bound the first term, we notice that 
\begin{eqnarray*}
\Q_2((o^*, 0) \leftrightarrow B(o,R)^c \times \mathbb{R}_+) \leq \Q_2((o^*, 0) \leftrightarrow B(o^*,R-1)^c \times \mathbb{R}_+).
\end{eqnarray*}
Hence, it remains to prove the following result.
\begin{lem} \label{dtd} There exists $\delta >0$, such that
$$ \Q_2((o^*, 0) \leftrightarrow B(o^*,R-1)^c \times \mathbb{R}_+) < \lambda ^ {\delta}.$$
\end{lem}
\begin{proof}
We follow the proof and notation in \cite{MVY}, let $R'=R-1$, and  $ L_1 = \lceil \lambda ^ {- \epsilon _1'/2} \rceil$. Then we define
\begin{align*}
\phi(T^*)= & \sum\limits_{i=1}^{R'} (2 \lambda)^i |S^*_{i,R',M}(T)|,\\
\psi(T^*)= & \sum\limits_{i=2}^{R'} (2 \lambda)^{2i} |T^*_{i,R',M}|,
\end{align*}
where $T^*$ is the tree $T$ rooted at $o^*$ and $T^*_{i,r,M}, S^*_{i,r,M}$ are defined in Lemma \ref{lup4}. 
We now define  
\begin{eqnarray*}
B^5_1  &= & \{ \phi (T^*) > \lambda ^{\epsilon_1 '}\}, \quad \quad B^5_2 = \{ \psi (T^*) > \lambda ^{\epsilon_1 '}\},\\
B^5_3&=& (B^5_1 \cup B^5_2) ^c \cap \left \{ \{o^*\} \times [0, L_1] \leftrightarrow \left(\bigcup_{i=1}^{R'}S^*_{i,R',M} \right) \times \mathbb{R}_+ \right \},\\
B^5_4&=& (B^5_1 \cup B^5_2) ^c \cap \{ \exists z : d(o^*,z ) \geq 2, \{o^*\} \times [0, L_1] \leftrightarrow (z,s) \leftrightarrow \{o^*\} \times [s, \infty)  \\
 & & \hspace{9 cm} \textrm{ inside } T(z) \cap T_{R',M} \},\\
B^5_5&=& \{ B(o^*,1) \times \{0\} \leftrightarrow B(o^*,1) \times \{L_1\} \textrm{ inside }  B(o^*,1)\}.
\end{eqnarray*}
It is explained in  \cite{MVY} that 
$$ \{ (o^*,0) \leftrightarrow B(o^*,R')^c \times \mathbb{R}_+\} \subset \bigcup _{i=1} ^5 B^5_i .$$
{\bf Event} $ B^5_1$.  Similarly to $B_1^4$, using Lemma \ref{lup4}, we have 
\begin{eqnarray*}
\mathbb{E}_{\mathbb{Q}_2} (\phi(T^*) ) &\lesssim &  \sum \limits_{i=1}^{R'-1} (2 \lambda) ^i  (\log M)^{i-1} M'M^{-1} + (2 \lambda) ^{R'} M' M^{-1}(\log M) ^{R'}\\
& \lesssim & \lambda^{3/4}. 
\end{eqnarray*}
Then using Markov's inequality we get
$$\mathbb{Q}_2(B^5_1)= \kO( \lambda^{3/4 - \varepsilon_1'}) = \kO(\lambda^{1/4}), $$
since $\varepsilon_1' \leq 1/2$.

\vspace{0.2 cm}
\noindent {\bf Event} $ B^5_2$. We have
\begin{eqnarray*}
\mathbb{E}_{\mathbb{Q}_2} (\psi(T^*) ) & \lesssim &  M' \sum \limits_{i=2}^{R'} (2 \lambda) ^{2i } (\log M) ^{(i-1)} \\
& \lesssim & \lambda^{3/2}.
\end{eqnarray*}
Then it follows from Markov's inequality that 
$$\mathbb{Q}_2 (B^5_2 ) =  \kO (\lambda^{3/2 - \varepsilon_1'}) = \kO(\lambda).$$
The events $B^5_3, B^5_4, B^5_5$ are exactly the same as in \cite{MVY}.  
\end{proof}
We now conclude the proof of \eqref{uopr}. By \eqref{potr} and the estimates of $(B^4_i)_{i \leq 5}$, we obtain 
\begin{eqnarray*}
\pp((o,0) \leftrightarrow B(o,R)^c \times \mathbb{R}_+ ) & \lesssim & \pp(B^4_6) \\
& \lesssim & \lambda^{1 + 2/ \psi} (\lambda^{\delta} + |\log \lambda|^{-1/\psi}) \\
& = & \kO(\lambda^{1 + 2/ \psi} |\log \lambda|^{-1/\psi}),
\end{eqnarray*}
which proves the desired result. \hfill $\square$

\vspace{0.5 cm}
\noindent {\it Proof of Lemma \ref{lup1}}.  Since we fix $r$ and $M$, we omit it in the notation. Let us define for $1 \leq i \leq r$
\begin{align*}
f_i(x) = \E(|T^v_{i}| 1(\deg(v) \leq M) \mid x_v =x),
\end{align*}
where $v$ is any vertex different from the root $o$. For $1 \leq i \leq r-1$,
\begin{align*}
f_{i+1}(x) & = \E(|T^v_{i+1}| 1(\deg(v) \leq M) \mid x_v =x) \\
           & = \sum \limits_{k \leq M-m-1} \E(|T^v_{i+1}| 1(\deg(v) =m+1+k) \mid x_v =x) \\
           & = \sum \limits_{k \leq M-m-1} \E(|T^v_{i+1}|  \mid x_v =x, \deg(v) =m+1+k) p(k,x) \\
           & = \sum \limits_{k \leq M-m-1} \sum \limits_{j =1 }^{ m+k} \E(|T^{(v,j)}_{i}| 1(\deg((v,j)) \leq M)  \mid x_v =x, \deg(v) =m+1+k) p(k,x), 
\end{align*}
where for the last line, we used \eqref{tiv}.

On the event  $\{ x_v =x,  \deg(v)=m+1+k \}$, $x_{(v,1)}, \ldots, x_{(v,m)}$ are uniformly distributed on $[0,x]$ and  $x_{(v,m+1)}, \ldots, x_{(v,m+k)}$  are distributed on $[x,1]$ with density $\frac{\psi y^{\psi -1}}{1-x^{\psi}} dy$. Therefore
 \begin{align} \label{eqf}
f_{i+1}(x)  & = \sum \limits_{k \leq M-m-1} \left( \frac{m}{x} \int_0^x f_i(y) dy + \frac{k}{1-x^{\psi}} \int_x^1 \psi y^{\psi-1}f_i(y) dy  \right)  p(k,x) \notag \\
 & \leq \frac{m}{x} \int_0^x f_i(y) dy  +  \frac{F(M,x)}{1-x^{\psi}} \int_x^1 \psi y^{\psi-1} f_i(y) dy,
\end{align}
where 
$$F(M,x)=\sum \limits_{k \leq M-m-1} k p(k,x).$$
 Moreover, it follows from \eqref{t1v} that
\begin{align} \label{f1}
f_1(x) = \E((\deg(v)-1)1(\deg(v) \leq M) \mid x_v=x) = m+ F(M,x).
\end{align}
Hence by \eqref{tpk1},
\begin{align} \label{fmx} 
f_1(x) \lesssim F^*(M,x):=  (M \wedge x^{-\psi} )^{a+1} x^{a\psi} +1.
\end{align}
After some simple computations, we have
\begin{align} \label{tf1}
\frac{1}{x} \int_0^x F^*(M,y) dy \lesssim F^* (M,x)\log M. 
\end{align}
and
\begin{align} \label{tf2}
\frac{1}{1-x^{\psi}} \int_x^1 \psi y^{\psi-1} F^*(M,y) dy \lesssim \log M. 
\end{align}
From \eqref{eqf}, \eqref{f1}, \eqref{fmx}, \eqref{tf1} and \eqref{tf2},  we can prove by induction that  for  $1 \leq i \leq r$,
\begin{align} \label{fix}
f_i(x) \leq C^i(\log M)^{i-1}F^* (M,x).
\end{align}
for some constant $C>0$. Similarly to \eqref{eqf}, we also have 
\begin{align*}
\E(|T_i| \mid x_o=x) \leq  \frac{m}{x} \int_0^{x} f_{i-1}(y) dy + \sum \limits_{k  \geq 0} \frac{k p(k,x)}{1-x^{\psi}} \int_{x}^1 \psi y^{\psi-1}f_{i-1}(y) dy. 
\end{align*}
It follows from this estimate, \eqref{tkpk} and \eqref{fix}  that
\begin{align*}
\E(|T_i| \mid x_o=x) \leq C^i( \log M)^{i-1}(F^*(M,x) + x^{-\psi}).
\end{align*}
Hence using that $x_o \sim \kU([0,1])^{\chi}$ with $\chi=1/(\psi +1)$, we get that for $i \leq r$
\begin{eqnarray*}
\E(|T_i|) &= &(\psi+1) \int _0^1 \E(|T_i| \mid x_o = x) x^{\psi} dx  \\
&\leq & C^i( \log M)^{i-1}.
\end{eqnarray*}
Now, to estimate $\E(|S_i|)$, we  define for $1 \leq i \leq r-1$ 
\begin{align*}
g_i(x)= \E(|S^v_i| 1(\deg(v) \leq M) \mid x_v=x).
\end{align*}
As for $f_i(x)$, we also have for $1 \leq i \leq r-2$,
\begin{align} \label{rg}
g_{i+1}(x) \leq \frac{m}{x} \int_0^x g_i(y) dy + \frac{F(M,x)}{1-x^{\psi}} \int_x^1 \psi y^{\psi -1} g_i(y) dy,
\end{align}
and by \eqref{s1v},
\begin{align} \label{g1}
g_1(x) & = \E(|\{w: w \textrm{ is a child of $v$ with }  \deg(w) >M \}| 1(\deg(v) \leq M) \mid x_v=x ) \notag\\
& \leq \frac{m}{x} \int_0^x g_0(y)  dy + \frac{F(M,x)}{1-x^{\psi}} \int_x^1 \psi y^{\psi -1} g_0(y) dy,  
\end{align}
where 
\[g_0(y)= \pp(\deg(w)>M \mid x_w =y).\]
It follows from \eqref{pkx} that 
\begin{align*}
g_0(y) & \asymp y^{a \psi}\sum \limits_{k>M} k^{a-1} (1- y^{\psi})^k \\
& \lesssim y^{a \psi} \int_M^{\infty} t^{a-1} \exp(-t y^{\psi}) dt  \\
&  \asymp \int_{My^{\psi}}^{\infty} t^{a-1} \exp(-t ) dt  \\
& \lesssim    \exp(-M y^{\psi}/2). 
\end{align*}
On the other hand, if $1> My^{\psi} >1/2$ then
\begin{align*}
g_0(y) & \gtrsim y^{a \psi}\sum \limits_{k=M+1}^{2M} k^{a-1} (1- y^{\psi})^k \\
& \gtrsim 1.
\end{align*}
Therefore
\begin{align} \label{goy}
1(1/2 < My^{\psi}<1) \lesssim g_0(y) \lesssim \exp(-M y^{\psi}/2). 
\end{align}
Let us define for $M\geq 2$
\begin{align} 
\alpha & = p_{L,M}(x) =  \pp(\deg((v,1)) >M \mid x_v =x) =\frac{1}{x} \int_0^x g_0(y) dy \label{dplm} \\ 
 \beta &= p_{R,M}(x) = \pp(\deg((v,m+1)) >M \mid x_v =x)  = \frac{1}{1-x^{\psi}} \int_x^1 \psi y^{\psi-1}g_0(y) dy.  \label{dprm}
\end{align}
Then using \eqref{goy} we obtain that
\begin{align}
(Mx^{\psi})^{-1/\psi} 1(Mx^{\psi} \geq 1 ) \lesssim \alpha  &\lesssim 1(Mx^{\psi}<1) + (Mx^{\psi})^{-1/\psi} 1(Mx^{\psi} \geq 1 ) \label{plm} \\
\beta &  \lesssim  M^{-1}\exp(-M x^{\psi}/2). \label{prm}
\end{align}
Define 
$$G^*(M,x)= 1(Mx^{\psi} <1) + (M x^{\psi})^{-1/\psi} 1(Mx^{\psi} \geq 1).$$
Then 
$$\beta F(M,x)  = \kO(G^*(M,x)).$$
Therefore using this, \eqref{g1}, \eqref{plm} and \eqref{prm} we get 
\begin{align} \label{bg1}
g_1(x) = \kO(G^*(M,x)).
\end{align}
Furthermore, 
\begin{align}
\frac{1}{x} \int_0^x G^*(M,y) dy & = \kO(  (\log M) G^*(M,x) ), \label{pgso}\\
\frac{1}{1-x^{\psi}} \int_x^1 \psi y^{\psi-1} G^*(M,y) dy & = \kO(M^{-1} 1(Mx^{\psi} <1) + M^{-1/\psi} 1(Mx^{\psi} \geq 1)) \label{pgsm}.
\end{align}
Hence
\begin{align*}
\frac{m}{x} \int_0^x G^*(M,y) dy + \frac{F(M,x)}{1-x} \int_x^1 G^*(M,y) dy = \kO( (\log M) G^*(M,x)).
\end{align*}
From this estimate, \eqref{rg} and \eqref{bg1}, we can prove by induction  that  for $1 \leq i \leq r-1$
\begin{align}
g_i(x) = \kO((\log M)^{i-1} G^*(M,x)). \label{pgix}
\end{align}
We now have
\begin{align*}
\E(|S_i| \mid x_o) &\leq  \frac{m}{x_o} \int_0^{x_o} g_{i-1}(y) dy + \sum \limits_{k  \geq 0} \frac{k p(k,x_o)}{1-x_o^{\psi}} \int_{x_o}^1 \psi y^{\psi -1} g_{i-1}(y) dy. \\
& = \kO \left( \frac{1}{x_o} \int_0^{x_o} g_{i-1}(y) dy +  \frac{x_o^{-\psi}}{(1-x_o^{\psi})} \int_{x_o}^1 \psi y^{\psi -1}g_{i-1}(y) dy \right)\\
& = \kO \left( (\log M)^{i-1} [ G^*(M,x_o) +  x_o^{-\psi}(M^{-1}1(Mx^{\psi}<1) + M^{-1/ \psi} 1(Mx^{\psi}\geq 1))] \right).
\end{align*}
Finally,
\begin{align*}
 \E[G^*(M,x_o) +  x_o^{-\psi}(M^{-1}1(Mx^{\psi}<1) + M^{-1/ \psi} 1(Mx^{\psi}\geq 1))]   = \kO(M ^{-1/ \psi}).
\end{align*}
Then the result follows from  the  last two estimates. 
\hfill $\square$

\vspace{0.2 cm}
\noindent {\it Proof of Lemma \ref{lup2}}. We also omit here $r$ and $M$ in the notation. Then
\begin{align*}
&\E(|S_1| 1(|S_1| \geq 2) \mid \deg(o) \leq M, x_o=x) \\
& \leq \sum \limits_{k \leq M} \E(|S_1| 1(|S_1| \geq 2) \mid x_o=x, \deg(o)=m+k) p(k,x).  
\end{align*}
Conditionally on the event  $\{x_o=x,  \deg(o)=m+k\}$, $|S_1|$ has the same distribution as 
\begin{align*}
X= X_1 + \ldots + X_m + Y_1 + \ldots + Y_k,
\end{align*}
where $(X_i)$ and $(Y_j)$ are independent Bernoulli random variables with mean $\alpha = p_{L,M}(x)$ and $\beta = p_{R,M}(x)$ respectively, as defined  in \eqref{dplm} and \eqref{dprm}.  Then
\begin{align*}
\E(X1(X \geq 2)) & =  \E(X) - \pp(X=1) \\
& = m \alpha + k \beta - m \alpha (1- \alpha)^{m-1} (1- \beta)^k - k \beta (1- \alpha)^{m} (1- \beta)^{k-1}  \\
& \leq (m \alpha + k \beta)^2 \\
& \leq 2(m^2  \alpha^2 + k^2 \beta^2).
\end{align*}
Therefore
\begin{align*}
\E(|S_1| 1(|S_1| \geq 2) \mid \deg(o)\leq M, x_o=x)&  \leq \sum \limits_{k \leq M} ( 2m^2 \alpha^2 + 2k^2 \beta ^2) p(k,x) \\
& \leq 2 m^2 \alpha^2 + 2 \beta^2  \sum \limits_{k \leq M} k^2 p(k,x). 
\end{align*}
We now take the expectation with respect to $x_o$. Since it has density $(\psi+1) x^{\psi} dx$ on $[0,1]$, the expectation of the first term is of order  
\begin{align*}
 \int_0^1 \alpha^2 x^{\psi} dx & \lesssim  \int_0^1 [1(Mx^{\psi}<1) + (Mx^{\psi})^{-2/\psi} 1(Mx^{\psi} \geq 1 )] x^{\psi} dx \\
 & \lesssim  \int_0^{M^{-1/\psi}} x^{\psi} dx  + M^{-2/\psi}\int_{M^{-1/\psi}}^1  x^{\psi -2} dx \\
 &=\kO( M^{-1-1/\psi} \log M).
\end{align*}
By \eqref{tpk2} and \eqref{prm}, the expectation of the second term is equivalent to
\begin{align*}
 \int _0^1 \beta^2 (M \wedge x^{-\psi})^{a+2} x^{(a+1)\psi} dx &\lesssim  M^{-2} \int _0^1 \exp(-M x^{\psi}) (M \wedge x^{-\psi})^{a+2} x^{(a+1)\psi} dx \\
 &\lesssim  M^{-2}  \int_0^{M^{-1/\psi}}  M^{a+2} x^{(a+1)\psi} dx + M^{-2}\int_{M^{-1/\psi}}^1 \exp(-M x^{\psi}) x^{-\psi} dx\\
&\lesssim M^{-1-1/\psi} + M^{-2}\int_1^M e^{-u} \left( \frac{M}{u}\right) M^{-1/\psi} u^{-1+ 1/\psi} du \\
&  =\kO(M^{-1-1/\psi}),
\end{align*}
where for the third line we used the change of variables $u=Mx^{\psi}$. Combining these last two estimates, we get the lemma.
\hfill $\square$

\vspace{0,2 cm}
\noindent {\it Proof of Lemma \ref{lup3}}.
With the same $\alpha$ and $\beta$ as in the previous lemma, we have
\begin{align*}
& \pp(\deg(o) \leq M, |S_1|=1 \mid x_o=x) \\
& = \sum \limits_{k \leq M-m } \left[  
m \alpha (1- \alpha)^{m-1} (1- \beta)^k + k \beta (1- \alpha)^{m} (1- \beta)^{k-1} \right] p(k,x) \\
& \leq \sum \limits_{k \leq M } \left[ m \alpha + k \beta \right] p(k,x) \\
& \lesssim \alpha + \beta (M \wedge x^{- \psi})^{a+1} x^{a \psi}.
\end{align*}
Then we take expectation in $x_o$ and get
\begin{align*}
& \pp(\deg(o) \leq M, |S_1|=1) = \kO(M^{-1/\psi}),
\end{align*}
which proves the first estimate. For the second one, we note that $(1- \beta)^k \asymp 1$ for $1 \leq k \leq M$ since $\beta = \kO(M^{-1})$. Hence  using \eqref{tpk0}, we get
\begin{align*}
 \pp(\deg(o) \leq M, |S_1|=1 \mid x_o=x) \gtrsim \alpha (1- \alpha)^{m-1} (M \wedge x^{- \psi})^{a} x^{a \psi}.
\end{align*}
It follows from \eqref{plm} and the fact that $\alpha \leq 1$ that  there is a positive constant $C$ independent of $x$, such that 
$$\alpha \leq 1(Mx^{\psi}<1) + (C(Mx^{\psi})^{-1/\psi} \wedge 1) 1(Mx^{\psi} \geq 1 )$$
Hence, there is a positive constant $c=c(C)$, such that
\begin{align} \label{balp}
 (1- \alpha)^{m-1} \geq c1(M x^{\psi} \geq 1/c).
\end{align} 
 Therefore   
\begin{align} \label{lo}
\pp(\deg(o) \leq M, |S_1|=1) & = (\psi +1) \int_0^1 \pp(\deg(o) \leq M, |S_1|=1 \mid x_o=x) x^{\psi} dx \notag \\
& \gtrsim \int_0^1 \alpha (1- \alpha)^{m-1} (M \wedge x^{-\psi})^a x^{a \psi} x^{\psi} dx \notag  \\
& \gtrsim   \int_{(cM)^{-1/ \psi}}^1  (M  x^{\psi})^{-1/ \psi} x^{\psi} dx  \notag \\
& \gtrsim M^{-1/ \psi}.
\end{align}
On the other hand,
\begin{align*}
\pp(\deg(o) \leq M, S_1=\{o^*\}, \deg(o^*) \geq M' \mid x_o=x) & \leq \sum \limits_{k \leq M-m } \left[  
m \alpha' + k \beta'  \right] p(k,x). \\
& \leq m \alpha' + \beta' F(M,x),
\end{align*}
with $$\alpha'=p_{L,M'}(x)\qquad \textrm{and} \qquad \beta'= p_{R,M'}(x).$$ Similarly to the calculus for $\alpha$ and $\beta$, we get 
\begin{align*}
\int_0^1 \alpha' x^{\psi} dx = \kO ((M')^{-1/\psi}), 
\end{align*} 
and
\begin{eqnarray*}
\int_0^1 \beta'  F(M,x) x^{\psi} &\lesssim & \frac{1}{M'} \int_0^1 e^{-M'x^{\psi}/2} (M\wedge x^{-\psi})^{a+1}x^{(a+1)\psi} dx  \\
&\lesssim & (M')^{-1-1/\psi}.
\end{eqnarray*}
 Hence 
\begin{align} \label{upo}
\pp(\deg(o) \leq M, S_1=\{o^*\}, \deg(o^*) \geq M') & =\kO \left( \int_0^1 (m \alpha' + \beta' F(M,x)) x^{\psi} dx \right) \notag \\
& = \kO \left( (M')^{-1/\psi} \right).
\end{align}
The result now follows from \eqref{lo} and \eqref{upo}.
\hfill $\square$

\vspace{0,2 cm}
\noindent {\it Proof of Lemma \ref{lup4}.}
 We start with the estimate for $|T^*_{i,r,M}|$.  Let us define 
\begin{align*}
A= \{\deg(o) \leq M\} \quad \textrm{ and }\quad  B=\{S_{1,r,M} = \{o^*\}, \deg(o^*) = M'\}.
\end{align*}
Similarly to the previous lemma, 
\begin{align} 
& \pp(\deg(o) \leq M, S_{1,r,M} = \{o^*\}, \deg(o^*) = M' \mid x_o=x) \notag \\
 & = \sum \limits_{k \leq M-m }  \left[m \alpha'' (1- \alpha)^{m-1}  (1-\beta)^k  + k \beta'' (1- \alpha)^{m}  (1-\beta)^{k-1} \right] p(k,x) \label{dgo}\\
 & \gtrsim  \alpha'' (1- \alpha)^{m-1} (M \wedge x^{- \psi})^a x^{a \psi}, \label{dtnd}
\end{align}
where 
\begin{align*}
\alpha'' := \pp(\deg((0,1)) =M' \mid x_o =x) = \frac{1}{x} \int_0^x p(M',y) dy
\end{align*}
and 
\begin{align*}
\beta''  := \pp(\deg((0,m+1)) =M' \mid x_o =x)= \frac{1}{1-x^{\psi}} \int_x^1 p(M',y)\psi y^{\psi -1} dy. 
\end{align*}
Similarly to the calculus for $\alpha$ and $\beta$, we have
\begin{align}
(M')^{-1-1/\psi} x^{-1} 1(M'x^{\psi} \geq 1) \lesssim \alpha''  \lesssim & x^{\psi}(M'x^{\psi})^{a-1} 1(M'x^{\psi} <1) \label{app} \\
& + (M')^{-1-1/\psi} x^{-1} 1(M'x^{\psi} \geq 1),  \notag \\
\beta''  \lesssim & (M')^{-2} e^{-M'x^{\psi}/2}. \label{bpp}
\end{align}
Hence, it follows from \eqref{balp}, \eqref{dtnd} and \eqref{app} that 
\begin{align*}
\pp(A \cap B \mid x_o=x) \gtrsim (M')^{-1-1/\psi} x^{-1} 1(Mx^{\psi} \geq 1/c).
\end{align*}
Therefore 
\begin{align} \label{lab}
\pp(A \cap B) \gtrsim (M')^{-1-1/\psi}.
\end{align}
We now prove that 
\begin{align*}
\E(|T^*_{i,r,M}| 1(A)1(B)) \leq C^i(\log M)^{i-1} M' (M')^{-1-1/\psi} .
\end{align*}
For $i=1$,  observe that on $A \cap B$,  $|T^*_{1,r,,M}|= \deg(o^*) = M'$. It follows from \eqref{dgo} that 
\begin{align*}
\pp(A \cap B \mid x_o=x) \leq m\alpha''+ \beta'' F(M,x).
\end{align*}
Then using that
\begin{align*}
 \int_0^1 \alpha'' x^{\psi} dx & = \kO((M')^{-1-1/\psi})  \\
\int_0^1 \beta'' F(M,x) x^{\psi } dx & =\kO((M')^{-2-1/\psi}),
\end{align*}
we obtain 
\begin{align} \label{pab}
\pp(A \cap B) = \kO((M')^{-1-1/\psi}). 
\end{align}
Therefore 
\begin{align*}
\E(|T^*_{1,r,M}| 1(A)1(B)) = \kO(M'(M')^{-1-1/\psi}).
\end{align*}
For $i \geq 2$, we notice that
\begin{align} \label{tsrm}
|T^*_{i,r,M}| & = |T^{o^*}_{i,r,M}| + \sum \limits_{(0,j) \neq o^* }|T^{(0,j)}_{i-2,r-2,M}|,
\end{align}
with the convention $|T_0|=1$.
We see that for any $1 \leq i \leq r$
\begin{align}
h_{1,i}(x)& = \E \left(\sum \limits_{(0,j) \neq o^* } |T^{(0,j)}_{i,r,M}| 1(A) 1(B) \mid x_o =x \right) \notag \\
& = \sum \limits_{k=0 }^{M-m} \E \left( \sum \limits_{(0,j) \neq o^*} |T^{(0,j)}_{i,r,M}|  1(B) \mid  x_o =x, \deg(o)=m+k \right)p(k,x) \notag \\
& = \sum \limits_{k=0 }^{M-m} \E \left( \sum \limits_{1 \leq j\neq s \leq m+k } |T^{(0,j)}_{i,r,M}|  1(B_{s,k}) \mid  x_o =x, \deg(o)=m+k \right)p(k,x), \label{h1i}
\end{align}
where 
$$B_{s,k}= \{\deg(o)=m+k\} \cap \{\deg((0,s))=M', \deg((0,j)) \leq M \,\, \forall \, j \neq s\}.$$
If $j \neq s$, then 
\begin{align}
& \E \left( |T^{(0,j)}_{i,r,M}|  1(B_{s,k}) \mid x_{(0,1)}, \ldots, x_{(0,m+k)} \right) \notag \\
& \leq \E \left( |T^{(0,j)}_{i,r,M}|  1(\deg((0,j)) \leq M) \mid x_{(0,j)} \right) \pp(\deg((0,s))=M' \mid x_{(0,s)}) \notag  \\
& = f_i(x_{(0,j)}) p(M',x_{(0,s)}). \label{fip}
\end{align}
Now using this estimate and the fact that $x_{(0,1)}, \ldots, x_{(0,m)}$ are uniformly distributed in $[0,x_o]$ and that $x_{(0,m+1)}, \ldots, x_{(0,m+k)}$  are distributed in  $[x_o,1]$ with density $\frac{\psi y^{\psi-1}dy}{1-x_o^{\psi}}$, we deduce from \eqref{h1i} and \eqref{fip} that
\begin{align} 
h_{1,i}(x) & \leq  \left( \frac{1}{x} \int_0^x f_i(y)dy \right) \alpha''  \sum \limits_{k=0 }^{M-m} m^2  p(k,x) \label{h1} \\
& +   \left( \frac{1}{1-x^{\psi}} \int_x^1 \psi y^{\psi-1}f_i(y)dy \right) \beta''  \sum \limits_{k=0 }^{M-m} k^2 p(k,x) \notag \\
&+ \left[\left( \frac{1}{x} \int_0^x f_i(y)dy \right) \beta'' + \left( \frac{1}{1-x^{\psi}} \int_x^1 \psi y^{\psi-1} f_i(y)dy \right) \alpha''  \right] \sum \limits_{k=0 }^{M-m} (m k) p(k,x). \notag
\end{align}
Then using these estimates, \eqref{tf1}, \eqref{tf2}, \eqref{fix}, we obtain 
\begin{align*} 
h_{1,i}(x) & \lesssim (\log M)^i \left(  \alpha'' F^*(M,x) \sum \limits_{k \leq M} p(k,x)  + \beta'' \sum \limits_{k \leq M} k^2 p(k,x) + \left(\beta'' F^*(M,x) + \alpha'' \right)\right) \\
& = (\log M)^i (H_1(x)+H_2(x) +H_3(x)).
\end{align*}
After some computation, we get
\[\int_0^1 (H_1(x)+H_2(x)+H_3(x))x^{\psi} dx = \kO((M')^{-1-1/\psi} \log M)\]
Hence
\begin{align*} 
\int_0^1 h_{1,i}(x)x^{\psi} dx \leq   (C \log M)^{i+1} (M')^{-1-1/\psi}.
\end{align*}
Note that in \eqref{tsrm}, we need an estimate for $h_{1,i-2}(x)$:
\begin{align} \label{tph1}
\int_0^1 h_{1,i-2}(x)x^{\psi} dx \leq   (C \log M)^{i-1} (M')^{-1-1/\psi}.
\end{align}
If $j=s$, then 
\begin{eqnarray*}
& &\E \left( |T^{(0,j)}_{i,r,M}|  1(B_{j,k}) \mid x_{(0,1)}, \ldots, x_{(0,m+k)} \right) \\
& \leq & \E \left( |T^{(0,j)}_{i,r,M}|  1(\deg((v,j)) = M') \mid x_{(0,j)} \right)\\
& \leq &\left( \frac{m}{x_{(0,j)}} \int _0^{x_{(0,j)}}  f_{i-1}(y) dy + \frac{M'-m}{1-x_{(0,j)}^{\psi}} \int _{x_{(0,j)}}^1 \psi x^{\psi -1} f_{i-1}(y) dy \right) p(M',x_{(0,j)})\\
& \leq & C^i(\log M)^{i-1} M'p(M',x_{(0,j)}).
\end{eqnarray*}
 Hence 
\begin{eqnarray*}
h_{2,i}(x) & = & \E \left(|T^{o^*}_{i,r,M}|  \mid x_o=x  \right) \\
& = &\sum \limits_{k=0 }^{M-m} \E \left( \sum \limits_{j  =1}^{m+k} |T^{(0,j)}_{i,r,M}|  1(B_{j,k}) \mid  x_o =x, \deg(o)=m+k \right)p(k,x) \\
&  \leq & C ^i( \log M)^{i-1} M' \sum \limits_{k=0 }^{M-m} \left( m \alpha'' + k \beta''\right)   p(k,x).
\end{eqnarray*}
Therefore using the same estimate as \eqref{pab}, we get
\begin{align} \label{th2}
\int_0^1h_{2,i}(x) x^{\psi} dx    \leq C^i(\log M)^{i-1} M'(M')^{-1-1/\psi}.
\end{align}
From \eqref{lab}, \eqref{tsrm}, \eqref{tph1} and \eqref{th2} we deduce that 
\[\E_{\mathbb{Q}_2}(|T^*_{i,r,M}|) \leq C^i(\log M)^{i-1} M'.\]
We  estimate $|S^*_{i,r,M}|$ by the same way. First, as for $h_{1,i}(x)$, we have
\begin{eqnarray*}
l_{1,i}(x)& = & \E \left(\sum \limits_{(0,j) \neq o^* } |S^{(0,j)}_{i,r,M}| 1(A) 1(B) \mid x_o =x \right) \\
& \leq &   \left( \frac{1}{x} \int_0^x g_i(y)dy \right) \alpha''  \sum \limits_{k=0 }^{M-m} m^2  p(k,x)  \\
& +  & \left( \frac{1}{1-x^{\psi}} \int_x^1 \psi y^{\psi-1}g_i(y)dy \right) \beta''  \sum \limits_{k=0 }^{M-m} k^2 p(k,x) \notag \\
&+ & \left[\left( \frac{1}{x} \int_0^x g_i(y)dy \right) \beta'' + \left( \frac{1}{1-x^{\psi}} \int_x^1 \psi y^{\psi-1} g_i(y)dy \right) \alpha''  \right] \sum \limits_{k=0 }^{M-m} (m k) p(k,x). \notag
\end{eqnarray*}
Then using \eqref{pgso}, \eqref{pgsm}, \eqref{pgix} and some computations, we get 
\begin{align*}
\int_0^1 l_{1,i-2}(x) x^{\psi} dx = \kO((\log M)^{i-1}M'M^{-1} (M')^{-1-1/\psi}).
\end{align*} 
We now consider
\begin{eqnarray*}
l_{2,i}(x)  & = & \E \left(|S^{o^*}_{i,r,M}|  \mid x_o=x  \right) \\
& = &\sum \limits_{k=0 }^{M-m} \E \left( \sum \limits_{j  =1}^{m+k} |S^{(0,j)}_{i,r,M}|  1(B_{j,k}) \mid  x_o =x, \deg(o)=m+k \right)p(k,x).
\end{eqnarray*}
On the other hand,
\begin{eqnarray*}
& & \E \left( |S^{(0,j)}_{i,r,M}|  1(B_{j,k}) \mid x_{(0,1)}, \ldots, x_{(0,m+k)} \right) \\
& \leq & \left( \frac{m}{x_{(0,j)}} \int _0^{x_{(0,j)}} g_{i-1}(y) dy + \frac{M'-m}{1-x_{(0,j)}^{\psi}} \int _{x_{(0,j)}}^1 g_{i-1}(y) \psi y^{\psi-1} dy \right) p(M',x_{(0,j)})\\
& = & \kO((\log M)^{i-1}   M'M^{-1} p(M',x_{(0,j)})).
\end{eqnarray*}
Here, we used \eqref{pgsm} to estimate the second term. Therefore
\begin{align*}
\int_0^1 l_{2,i}(x) x^{\psi} dx = \kO((\log M)^{i-1}M'M^{-1} (M')^{-1-1/\psi}).
\end{align*}
  We then conclude that   
\begin{eqnarray*}
\E(|S^*_{i,r,M}|1(A \cap B)) &= &\int_0^1 (l_{1,i-2}(x)+l_{2,i}(x)) x^{\psi} dx \\
&=& \kO((\log M)^{i-1}  M'M^{-1} (M')^{-1-1/\psi}).
\end{eqnarray*} 
 Therefore, by \eqref{lab}
 \begin{eqnarray*}
\E_{\Q_2}(|S^*_{i,r,M}|) &=& \kO((\log M)^{i-1}  M'M^{-1} ).
\end{eqnarray*} 
To estimate $\E_{\mathbb{Q}_1}(|T_{i,r,M}|)$, we use \eqref{lo} and the same argument as for $\E_{\mathbb{Q}_2}(|T^*_{i,r,M}|)$. More precisely, we replace $B$ by $\tilde{B} = \{ |S_{1,r,M}|=1\}$,  replace $B_{s,k}$ by
\begin{align*}
\tilde{B}_{s,k} = \{\deg(o)=m+k\} \cap \{\deg((0,s))>M, \deg((0,j)) \leq M \,\, \forall \, \, j \neq s\},
\end{align*}
 replace $\alpha''$ and   $\beta ''$ by  $\alpha$  and $\beta$ respectively. We now have
 \begin{align*}
 & \E(|T_{i,r,M}| 1(A) 1(\tilde{B}) \mid x_o=x) \\
  & =  \E \left( \sum \limits_{j =1 }^{\deg(o)} |T^{(0,j)}_{i-1,r-1,M}| 1(A) 1(\tilde{B}) 1(\deg((0,j)) \leq M) \mid x_o =x \right) \\
 & =  \sum_{k=0}^{M-m} \E\left( \sum \limits_{1 \leq j \neq s \leq m+k}   |T^{(0,j)}_{i-1,r-1,M}|1(\tilde{B}_{s,k}) \mid x_o =x, \deg(o)= m+k \right) p(k,x) \\
 & := \tilde{h}_{i-1}(x).
 \end{align*}
We estimate $\int_0^1 \tilde{h}_{i-1}(x) x^{\psi} dx$ by the same way as for $h_{1,i-2}(x)$, and get the desired result.
\hfill $\square$

 \begin{ack} \emph{
 I am  grateful to my  advisor Bruno Schapira for his help and many suggestions during the preparation of this work. I wish to thank also Daniel Valesin for  showing me a gap in a previous version, and the anonymous referee for carefully reading our manuscript and many valuable comments. }
 \end{ack}

\end{document}